\documentclass[12pt,final]{amsart}

\usepackage{xspace,amssymb,epsfig,euscript,eufrak,enumerate,amsmath,nicefrac}
\DeclareMathAlphabet{\mathbbm}{U}{bbm}{m}{n}
\usepackage{tikz,etex,xspace,afterpage}
 \usepackage{epstopdf,ifpdf,todonotes}
\usepackage[notref,notcite]{showkeys}
\usepackage[margin=3.2cm,footskip=25pt,headheight=20pt]{geometry}

\usepackage{fancyhdr,mathrsfs}
\usepackage{hyperref}

  \usepackage{pdfsync}

\begin{document}
\newtheoremstyle{all}%  |?name>
  {11pt}%      |?Space above>
  {11pt}%      |?Space below>
  {\slshape}%         |?Body font>
  {}%         |?Indent amount>|ntnote1
  {\bfseries}% |?Theorem head font>
  {}%        |?Punctuation after theorem head>
  {.5em}%     |?Space after theorem head>|ntnote2
  {}%         |?Theorem head spec (can be left empty, meaning `normal')>

\theoremstyle{all}

\newtheorem{itheorem}{Theorem}
\newtheorem{theorem}{Theorem}[section]
\newtheorem*{theoremfourfive}{Theorem 4.5}
\newtheorem*{proposition*}{Proposition}
\newtheorem{proposition}[theorem]{Proposition}
\newtheorem{corollary}[theorem]{Corollary}
\newtheorem{lemma}[theorem]{Lemma}
\newtheorem{assumption}[theorem]{Assumption}
\newtheorem{definition}[theorem]{Definition}
\newtheorem{ques}[theorem]{Question}
\newtheorem{conj}[theorem]{Conjecture}

\theoremstyle{remark}
\newtheorem{remark}[theorem]{Remark}
\newtheorem{example}[theorem]{Example}

\newcommand{\nc}{\newcommand}
\newcommand{\renc}{\renewcommand}
\newcounter{subeqn}
\renewcommand{\thesubeqn}{\theequation\alph{subeqn}}
\newcommand{\subeqn}{%
  \refstepcounter{subeqn}% Step subequation number
  \tag{\thesubeqn}% Label equation
}\makeatletter
\@addtoreset{subeqn}{equation}
\newcommand{\newseq}{%
  \refstepcounter{equation}% Step subequation number
}
  \nc{\kac}{\kappa^C}
\nc{\alg}{A}
\nc{\Lco}{L_{\la}}
\nc{\qD}{q^{\nicefrac 1D}}
\nc{\ocL}{M_{\la}}
\nc{\excise}[1]{}
\nc{\Dbe}{D^{\uparrow}}
\nc{\Dfg}{D^{\mathsf{fg}}}

\nc{\op}{\operatorname{op}}
\nc{\Sym}{\operatorname{Sym}}
\nc{\tr}{\operatorname{tr}}
\newcommand{\Mirkovic}{Mirkovi\'c\xspace}
\nc{\tla}{\mathsf{t}_\la}
\nc{\llrr}{\langle\la,\rho\rangle}
\nc{\lllr}{\langle\la,\la\rangle}
\nc{\K}{\mathbbm{k}}
\nc{\Stosic}{Sto{\v{s}}i{\'c}\xspace}
\nc{\cd}{\mathcal{D}}
\nc{\cT}{\mathcal{T}}
\nc{\vd}{\mathbb{D}}
\nc{\lift}{\gamma}
\nc{\cox}{h}
\nc{\Aut}{\operatorname{Aut}}
\nc{\R}{\mathbb{R}}
\renc{\wr}{\operatorname{wr}}
  \nc{\Lam}[3]{\La^{#1}_{#2,#3}}
  \nc{\Lab}[2]{\La^{#1}_{#2}}
  \nc{\Lamvwy}{\Lam\Bv\Bw\By}
  \nc{\Labwv}{\Lab\Bw\Bv}
  \nc{\nak}[3]{\mathcal{N}(#1,#2,#3)}
  \nc{\hw}{highest weight\xspace}
  \nc{\al}{\alpha}
\numberwithin{equation}{section}
\renc{\theequation}{\arabic{section}.\arabic{equation}}
  \nc{\be}{\beta}
  \nc{\bM}{\mathbf{m}}
  \nc{\Bu}{\mathbf{u}}
\nc{\Ba}{\mathbf{a}}

  \nc{\bkh}{\backslash}
  \nc{\Bi}{\mathbf{i}}
  \nc{\Bm}{\mathbf{m}}
  \nc{\Bj}{\mathbf{j}}
 \nc{\Bk}{\mathbf{k}}
\newcommand{\bS}{\mathbb{S}}
\newcommand{\bT}{\mathbb{T}}
\newcommand{\bt}{\mathbbm{t}}

\nc{\bd}{\mathbf{d}}
\nc{\D}{\mathcal{D}}
\nc{\mmod}{\operatorname{-mod}}  
\newcommand{\red}{\mathfrak{r}}

\nc{\RAA}{R^\A_A}
  \nc{\Bv}{\mathbf{v}}
  \nc{\Bw}{\mathbf{w}}
\nc{\Id}{\operatorname{Id}}
  \nc{\By}{\mathbf{y}}
\nc{\eE}{\EuScript{E}}
  \nc{\Bz}{\mathbf{z}}
  \nc{\coker}{\mathrm{coker}\,}
  \nc{\C}{\mathbb{C}}
\nc{\ab}{{\operatorname{ab}}}
\nc{\wall}{\mathbbm{w}}
  \nc{\ch}{\mathrm{ch}}
  \nc{\de}{\delta}
  \nc{\ep}{\epsilon}
  \nc{\Rep}[2]{\mathsf{Rep}_{#1}^{#2}}
  \nc{\Ev}[2]{E_{#1}^{#2}}
  \nc{\fr}[1]{\mathfrak{#1}}
  \nc{\fp}{\fr p}
  \nc{\fq}{\fr q}
  \nc{\fl}{\fr l}
  \nc{\fgl}{\fr{gl}}
\nc{\rad}{\operatorname{rad}}
\nc{\ind}{\operatorname{ind}}
  \nc{\GL}{\mathrm{GL}}
\newcommand{\arxiv}[1]{\href{http://arxiv.org/abs/#1}{\tt arXiv:\nolinkurl{#1}}}
  \nc{\Hom}{\mathrm{Hom}}
  \nc{\im}{\mathrm{im}\,}
  \nc{\La}{\Lambda}
  \nc{\la}{\lambda}
  \nc{\mult}{b^{\mu}_{\la_0}\!}
  \nc{\mc}[1]{\mathcal{#1}}
  \nc{\om}{\omega}
\nc{\gl}{\mathfrak{gl}}
  \nc{\cF}{\mathcal{F}}
%\renc{\thetheorem}{\arabic{theorem}}
 \nc{\cC}{\mathcal{C}}
  \nc{\Mor}{\mathsf{Mor}}
  \nc{\HOM}{\operatorname{HOM}}
  \nc{\Ob}{\mathsf{Ob}}
  \nc{\Vect}{\mathsf{Vect}}
\nc{\gVect}{\mathsf{gVect}}
  \nc{\modu}{\mathsf{-mod}}
\nc{\pmodu}{\mathsf{-pmod}}
  \nc{\qvw}[1]{\La(#1 \Bv,\Bw)}
  \nc{\van}[1]{\nu_{#1}}
  \nc{\Rperp}{R^\vee(X_0)^{\perp}}
  \nc{\si}{\sigma}
\nc{\sgns}{{\boldsymbol{\sigma}}}
  \nc{\croot}[1]{\al^\vee_{#1}}
\nc{\di}{\mathbf{d}}
  \nc{\SL}[1]{\mathrm{SL}_{#1}}
  \nc{\Th}{\theta}
  \nc{\vp}{\varphi}
  \nc{\wt}{\mathrm{wt}}
\nc{\te}{\tilde{e}}
\nc{\tf}{\tilde{f}}
\nc{\hwo}{\mathbb{V}}
\nc{\soc}{\operatorname{soc}}
\nc{\cosoc}{\operatorname{cosoc}}
 \nc{\Q}{\mathbb{Q}}

  \nc{\Z}{\mathbb{Z}}
  \nc{\Znn}{\Z_{\geq 0}}
  \nc{\ver}{\EuScript{V}}
  \nc{\Res}[2]{\operatorname{Res}^{#1}_{#2}}
  \nc{\edge}{\EuScript{E}}
  \nc{\Spec}{\mathrm{Spec}}
  \nc{\tie}{\EuScript{T}}
  \nc{\ml}[1]{\mathbb{D}^{#1}}
  \nc{\fQ}{\mathfrak{Q}}
        \nc{\fg}{\mathfrak{g}}
        \nc{\ft}{\mathfrak{t}}
  \nc{\Uq}{U_q(\fg)}
        \nc{\bom}{\boldsymbol{\omega}}
\nc{\bla}{{\underline{\boldsymbol{\la}}}}
\nc{\bmu}{{\underline{\boldsymbol{\mu}}}}
\nc{\bal}{{\boldsymbol{\al}}}
\nc{\bet}{{\boldsymbol{\eta}}}
\nc{\rola}{X}
\nc{\wela}{Y}
\nc{\fM}{\mathfrak{M}}
\nc{\fX}{\mathfrak{X}}
\nc{\fH}{\mathfrak{H}}
\nc{\fE}{\mathfrak{E}}
\nc{\fF}{\mathfrak{F}}
\nc{\fI}{\mathfrak{I}}
\nc{\qui}[2]{\fM_{#1}^{#2}}
\nc{\cL}{\mathcal{L}}
%\nc{\cO}{\mathcal{O}}
\nc{\ca}[2]{\fQ_{#1}^{#2}}
\nc{\cat}{\mathcal{V}}
\nc{\cata}{\mathfrak{V}}
\nc{\catf}{\mathscr{V}}
\nc{\hl}{\mathcal{X}}
\nc{\hld}{\EuScript{X}}
\nc{\hldbK}{\EuScript{X}^{\bla}_{\bar{\mathbb{K}}}}

%\nc{\hl}{\mathfrak{X}}
\nc{\pil}{{\boldsymbol{\pi}}^L}
\nc{\pir}{{\boldsymbol{\pi}}^R}
\nc{\cO}{\mathcal{O}}
\nc{\param}{\Pi}
\nc{\Ko}{\text{\Denarius}}
\nc{\Ei}{\fE_i}
\nc{\Fi}{\fF_i}
\nc{\fil}{\mathcal{H}}
\nc{\brr}[2]{\beta^R_{#1,#2}}
\nc{\brl}[2]{\beta^L_{#1,#2}}
\nc{\so}[2]{\EuScript{Q}^{#1}_{#2}}
\nc{\EW}{\mathbf{W}}
\nc{\rma}[2]{\mathbf{R}_{#1,#2}}
\nc{\Dif}{\EuScript{D}}%{\operatorname{Dif}}
\nc{\MDif}{\EuScript{E}}
\renc{\mod}{\mathsf{mod}}
\nc{\modg}{\mathsf{mod}^g}
\nc{\fmod}{\mathsf{mod}^{fd}}
\nc{\id}{\operatorname{id}}
\nc{\DR}{\mathbf{DR}}
\nc{\End}{\operatorname{End}}
\nc{\Fun}{\operatorname{Fun}}
\nc{\Ext}{\operatorname{Ext}}
\nc{\tw}{\tau}
\nc{\bbK}{\mathbb{K}}
\nc{\A}{\EuScript{A}}
\nc{\Loc}{\mathsf{Loc}}
\nc{\eF}{\EuScript{F}}
\nc{\LAA}{\Loc^{\A}_{A}}
\nc{\perv}{\mathsf{Perv}}
\nc{\gfq}[2]{B_{#1}^{#2}}
\nc{\qgf}[1]{A_{#1}}
\nc{\poly}{\mathscr{U}}
\nc{\qgr}{\qgf\rho}
\nc{\tqgf}{\tilde A}
\nc{\Tr}{\operatorname{Tr}}
\nc{\Tor}{\operatorname{Tor}}
\nc{\cQ}{\mathcal{Q}}
\nc{\st}[1]{\Delta(#1)}
\nc{\cst}[1]{\nabla(#1)}
\nc{\ei}{\mathbf{e}_i}
\nc{\Be}{\mathbf{e}}
\nc{\Hck}{\mathfrak{H}}
%\nc{\fH}{\Hck}
%\nc{\fM}{\mathfrak{M}}
\renc{\P}{\mathbb{P}}
\nc{\bbB}{\mathbb{B}}
\nc{\ssy}{\mathsf{y}}
\nc{\cI}{\mathcal{I}}
\nc{\cG}{\mathcal{G}}
\nc{\cH}{\mathcal{H}}
\nc{\coe}{\mathfrak{K}}
\nc{\pr}{\operatorname{pr}}
\nc{\bra}{\mathfrak{B}}
\nc{\ps}{s}
\nc{\cs}{\mathsf{s}}

\nc{\bps}{\mathbf{s}}
\nc{\rcl}{\rho^\vee(\la)}
\nc{\tU}{\mathcal{U}}
\nc{\dU}{{\stackon[8pt]{\tU}{\cdot}}}
\nc{\dT}{{\stackon[8pt]{\cT}{\cdot}}}
\nc{\Coulomb}{\mathfrak{A}}
\nc{\fCoulomb}{\mathfrak{A}'}
\nc{\Cspace}{\mathbb{X}}
\nc{\RHom}{\mathrm{RHom}}
\nc{\tcO}{\tilde{\cO}}
\nc{\Yon}{\mathscr{Y}}
\nc{\sI}{{\mathsf{I}}}
\nc{\sptc}{X_*(T)_1}
\nc{\DO}{u}
%\nc{\DO}{\mathscr{U}}
\nc{\spt}{\ft_1}
\nc{\Bpsi}{u}
\nc{\acham}{\eta}
\nc{\hyper}{\mathsf{H}}
\nc{\AF}{\EuScript{Fl}}
\nc{\VB}{\EuScript{X}}
\newcommand{\cOg}{\mathcal{O}_{\!\operatorname{g}}}
\newcommand{\dOg}{D_{\cOg}}
\newcommand{\preO}{p\cOg}
\newcommand{\dpreO}{D_{p\cOg}}
% Margin stuff from jason
%\oddsidemargin=0pt
%\evensidemargin=0pt
%\topmargin=0in
%\headheight=0pt
%\headsep=0pt
%\setlength{\textheight}{9in}
%\setlength{\textwidth}{6.5in}
\setcounter{tocdepth}{1}
\newcommand{\thetitle}{Representation theory of the cyclotomic
  Cherednik algebra via the Dunkl-Opdam subalgebra}

\newcommand{\theshorttitle}{Representation theory via the Dunkl-Opdam subalgebra}

\renc{\theitheorem}{\Alph{itheorem}}

\excise{
\newenvironment{block}
\newenvironment{frame}
\newenvironment{tikzpicture}
\newenvironment{equation*}
}

\baselineskip=1.1\baselineskip

 \usetikzlibrary{decorations.pathreplacing,backgrounds,decorations.markings,shapes.geometric}
\tikzset{wei/.style={draw=red,double=red!40!white,double distance=1.5pt,thin}}
\tikzset{awei/.style={draw=blue,double=blue!40!white,double distance=1.5pt,thin}}
\tikzset{bdot/.style={fill,circle,color=blue,inner sep=3pt,outer
    sep=0}}
\tikzset{dir/.style={postaction={decorate,decoration={markings,
    mark=at position .8 with {\arrow[scale=1.3]{>}}}}}}
\tikzset{rdir/.style={postaction={decorate,decoration={markings,
    mark=at position .8 with {\arrow[scale=1.3]{<}}}}}}
\tikzset{edir/.style={postaction={decorate,decoration={markings,
    mark=at position .2 with {\arrow[scale=1.3]{<}}}}}}
\tikzset{tstar/.style={fill=white,draw,star,star points=5,star point ratio=0.45,inner
          sep=3pt}}
\tikzset{ucircle/.style={fill=black,circle,inner sep=2pt}}
\begin{center}
\noindent {\large  \bf \thetitle}
\medskip

\noindent {\sc Ben Webster}\footnote{Supported by the NSF under Grant
  DMS-1151473 and by NSERC under a Discovery Grant. This research was supported in part by Perimeter Institute for Theoretical Physics. Research at Perimeter Institute is supported by the Government of Canada through the Department of Innovation, Science and Economic Development Canada and by the Province of Ontario through the Ministry of Research, Innovation and Science.
}\\  
Department of Pure Mathematics\\ University of Waterloo \&\\
Perimeter Institute for Mathematical Physics\\
Waterloo, ON, Canada\\
Email: {\tt ben.webster@uwaterloo.ca}
\end{center}
\bigskip
{\small
\begin{quote}
\noindent {\em Abstract.}
We give an alternate presentation of the cyclotomic rational Cherednik
algebra, which has the useful feature of compatibility with the
Dunkl-Opdam subalgebra.  This presentation has a diagrammatic flavor,
and it provides a simple explanation of several surprising facts about
this algebra.  It allows direct proof of the connection of category
$\cO$ to weighted KLR algebras, allows us to classify the simple
Dunkl-Opdam modules over the Cherednik algebra and provides an
algebraic construction of the KZ functor.  Furthermore, one of prime
motivations  for considering this approach is to provide a better framework for connecting Cherednik
algebras to Coulomb branches of 3-d gauge theories.
\end{quote}
}

\section{Introduction}
\label{sec:introduction}

In this paper we consider the rational Cherednik algebra $\mathsf{H}$
in the cyclotomic case, i.e. that of the complex reflection group
$G(\ell,1,n)$. This is an algebra with a quite rich and interesting
representation theory; this paper is dedicated to the proposition that
this representation theory can be understood more clearly by choosing
a different presentation.  In particular, we can classify the simple {\bf
  Dunkl-Opdam modules} over the Cherednik algebra in this case.  This
is the analogue for the Cherednik algebra of Gelfand-Tsetlin modules
over $U(\mathfrak{gl}_n)$; realizing both these algebras as Coulomb
branches makes this analogy manifest.  In fact, the approach we apply
here can be generalized to any rational Galois order, as we will show
in forthcoming work \cite{WebGT}.

In Section
\ref{sec:remind-cher-algebr}, we describe the presentation needed for
our results, and
prove that it gives the Cherednik algebra.
This presentation may not look obviously simpler than the familiar one
introduced by Etingof and Ginzburg \cite[(1.15)]{EtGi}, but it does
have a graphical calculus which allows it to be described in
terms of small local relations (much like the KLR algebras
\cite{KLI,Rou2KM}).  Furthermore, it has another dramatic advantage:
it contains a manifest polynomial subalgebra defined by Dunkl and
Opdam \cite[Def. 3.7]{DuOp}.  This subalgebra commutes with the
Euler element (unlike the usual polynomial subalgebras, where all
generators have weight $\pm 1$ in the Euler grading).   While
exploited profitably in earlier papers of Dunkl and Griffeth
\cite{DuGr,GrifII}, there is much more this subalgebra can tell us about
the representation theory of these algebras.  

In Section \ref{sec:weight-klr-algebr}, we turn to using this
presentation to study the representation theory of the Cherednik
algebra, using weight spaces for the Dunkl-Opdam polynomial
subalgebra.  This allows a new interpretation of previous work of the
author relating category $\cO$ of Cherednik algebras to weighted KLR
algebras \cite{WebRou}, a key step in proving Rouquier's conjecture on
the decomposition numbers of category $\mathcal{O}$ for Cherednik
algebras (this result was proved by other methods in
\cite{RSVV,LoVV}).  That work depended on a very indirect method using
uniqueness of highest weight covers, whereas using this new
presentation, it can be proven directly. Similarly, the
Knizhnik-Zamolodchikov functor of \cite{GGOR}, which had only been
constructed analytically before, can be realized as a sum of weight
spaces for the Dunkl-Opdam polynomial
subalgebra (in particular, we can define the KZ
functor over an arbitrary characteristic 0 field, not just $\C$).
These results are only valid in characteristic $0$, but this technique
is also promising for studying the Cherednik algebra and coherent
sheaves on Hilbert schemes in characteristic $p$.

In Section
\ref{sec:coulomb-branches}, we discuss the original motivation for
this presentation: to exhibit an isomorphism between the spherical
Cherednik algebra and the Coulomb branch of a certain 3-d gauge
theory.  While this paper was in preparation, this isomorphism was
proven independently by Kodera-Nakajima \cite{KoNa}.  This isomorphism
looks quite strange in the usual presentation of the Cherednik
algebra, and quite natural in the alternate one given here. It would
be quite interesting to find a geometric description of the Cherednik
algebra like the BFN construction of the Coulomb branch
\cite{NaCoulomb,BFN}, in terms of convolution in homology.

\section*{Acknowledgements}
We thank Stephen Griffeth for pointing out the connection of this
paper to his earlier work, Joel Kamnitzer and Ivan Losev for
discussions during the development of these ideas, Hiraku Nakajima for
a number of helpful comments on an early draft of this paper and
Alexander Braverman, Pavel Etingof and Michael Finkelberg for sharing
a preliminary version of their paper on related topics.

\section{An alternate presentation}
\label{sec:remind-cher-algebr}

Let $\K$ be a field of characteristic coprime to $\ell$  and $\zeta$
be a primitive $\ell$th root of unity in $\K$.  
Let $\bbK=\K[\hbar]$.  For most purposes, we
can take $\K=\C$ and $\zeta=e^{2\pi i/\ell}$.  

Let $\Gamma$ be the group of $n\times n$ monomial matrices with
entries given by $\ell$th roots of unity; this group is a wreath
product of $S_n$ with $\Z/\ell\Z$.  It's generated by the permutation
matrices (identified with $S_n$) and the matrices
$t_j=\operatorname{diag}(1,\dots, \zeta,\dots,1)$ with
$(t_j)_{jj}=\zeta$ and all other diagonal entries 1.

Fix parameters $k,h_1,\dots, h_{\ell-1}$, with the convention that
$h_0=h_\ell=0$ and let
\begin{equation}
p(u) = \sum_{s=1}^{\ell-1}
  \sum_{r=1}^{\ell-1} \zeta^{-rs} h_ru^s.\label{eq:p-def}
\end{equation}
We can equivalently fix the values 
\begin{equation}
\ps_m=p(\zeta^m)+m\hbar\text{  for } m=0,\dots, \ell-1.  \label{eq:s-def}
\end{equation}
We'll consider the cyclotomic rational Cherednik
  algebra $\mathsf{H}$ for $\Gamma$, generated over $\bbK[\Gamma]$ by
  two alphabets of commuting variables
  $x_1,\dots, x_n,$ $y_1,\dots, y_n$.  The former transform in the
  defining representation of $\Gamma$ and the latter in its dual.
  That is:
  \begin{equation}
    \label{eq:1}
    t_i x_j=\zeta^{\delta_{ij}} x_jt_i\qquad     t_i y_j=\zeta^{-\delta_{ij}} y_jt_i
  \end{equation}
  The
  final relation is 
\[[x,y]=\hbar \langle x,y\rangle -\sum_{s\in S} c_s \langle
x,\al_s\rangle \langle \al_s^\vee,y\rangle\cdot s  \]
We will use slightly different conventions here, following the
conventions of \cite[\S 2.1.3]{GoLo}, so these relations take the form:
\begin{align}
\label{xiyi} [x_i,y_i]&=\hbar +k\sum_{j\neq i} \sum_{p=0}^{\ell-1}  t_i^pt_j^{-p} (ij) +\sum_{s=1}^{\ell-1}\sum_{r=1}^{\ell}\zeta^{-rs}(h_r-h_{r-1})
             t_i^s\\
\notag             &= \hbar +k\sum_{j\neq i} \sum_{p=0}^{\ell-1}  t_i^pt_j^{-p} (ij) +p(t_i)-p(\zeta^{-1} t_i)\\
\label{xiyj}[x_i,y_j]&=-k \sum_{p=0}^{\ell-1}  \zeta^{p} t_i^pt_j^{-p}   (ij) \qquad \qquad \qquad \qquad \qquad (i\neq j)
\end{align}
Recall that this algebra contains the modified Dunkl-Opdam operators
\begin{align}
	\DO_i&=y_ix_i+ k\sum_{j> i} \sum_{p=0}^{\ell-1}  t_i^pt_j^{-p} (ij) + p(t_i)\\
	& =x_iy_i-k\sum_{j< i} \sum_{p=0}^{\ell-1}  t_i^pt_j^{-p} (ij) + p(\zeta^{-1} t_i) -\hbar
\end{align}
These differ from those defined in \cite[(2.17)]{GrifII} by
$z_i=\DO_i-p(\zeta^{-1}t_i)-\hbar$ and the reindexing of $1,\dots, n$
by $i\mapsto n-i+1$.  Note that since $t_i$ and $z_i$ generate a
commutative subalgebra, these elements $\DO_i$ commute with each other
and with $t_i$ (and generate the same subalgebra).  Accounting for
reindexing, the equation \cite[(3.5)]{GrifII} implies that if we let $r_j=(j,j+1)$:
\begin{equation}
	\DO_{r_j\cdot i}r_j-r_j\DO_{i} =
        \begin{cases}
          k\ell \pi_{j,j+1}&i=j\\
          -k\ell \pi_{j,j+1}&i=j+1\\
          0 & j\notin \{i,i-1\}
        \end{cases}
\label{degHeckeDO}
\end{equation}   
where $\pi_{j,m}= \frac{1}\ell\big (\sum_{p=0}^{\ell-1}
t_j^pt_{m}^{-p}\big) $ is the projection to the invariants of
$t_jt_m^{-1}$.  Let $\mathcal{D\!O}_n$ denote the algebra generated by
$\K \Gamma$ and $\DO_i$ modulo the relations \eqref{degHeckeDO}.

Consider the free $\bbK$ algebra $\tilde{A}$ generated by the group algebra
$\bbK\Gamma$ and the symbols
$\sigma, \tau,$ and $u_i$ for $i=1,\dots, n$.    We define $u_i,t_i\in \tilde{A}$ for any
$i\in\Z$ by the rule $u_i=u_{i-n}+\hbar$, and $t_i=t_{i-n}\zeta^{-1}$.  \begin{definition}
  We let $\alg$ be the quotient of this algebra by the relations
  \eqref{degHeckeDO} and:\newseq
  \begin{align*}
%    u_ir_j&=r_ju_{r_j\cdot
%                i}+k \cdot \al_j^\vee(\epsilon_i) \big (\sum_{p=0}^{\ell-1} t_j^pt_{j+1}^{-p}\big)&j&=1,\dots, n-1\label{degHecke}\subeqn\\
  %  t_ir_j&=r_jt_{r_j\cdot i}&j&=1,\dots, n-1\\
   u_it_j&=t_ju_i & i,j&\in \Z\subeqn \\
    \sigma r_{j-1}&=r_j\sigma& j&=2,\dots, n-1\label{isotope1}\subeqn\\
    \tau r_j&=r_{j-1}\tau& j&=2,\dots, n-1\label{isotope2}\subeqn\\
        \sigma^2 r_{n}&=r_1\sigma^2&\label{isotope3}\subeqn\\
    \tau^2 r_1&=r_{n}\tau^2& \label{isotope4}\subeqn\\
    \sigma\tau&=u_1-p(\zeta ^{-1} t_1)+\hbar \subeqn\label{sigtau}\\
    \tau\sigma&=u_n-p(t_n) \subeqn\label{tausig}\\
u_iu_j&=u_ju_i & i,j&\in \Z\label{u-commute}\subeqn\\
    u_i\sigma&=\sigma u_{i-1}& i&\in \Z\label{u-sig}\subeqn\\
    u_i\tau&=\tau u_{i+1}& i&\in \Z\label{u-tau}\subeqn\\
    t_i\sigma&=\sigma t_{i-1}& i&\in \Z\label{t-sig}\subeqn\\
    t_i\tau&=\tau t_{i+1}& i&\in \Z\label{t-tau}\subeqn\\
    \tau (1,2) \sigma &= \sigma (n-1,n) \tau +k \big(\sum_{p=0}^{\ell-1}\zeta^p t_n^pt_{1}^{-p}\big)\label{last}\subeqn
  \end{align*}
\end{definition}
\begin{remark}
 Note that these relations are closely related to those for the degenerate DAHA
  given in \cite[Def. 2.1]{BEF}, and should be regarded as a higher
  level version of this presentation.  
\end{remark}

We can represent these elements graphically as string diagrams on a
cylinder with a seam. We'll draw these on the page with the cylinder
cut along the seam.  The generators are: 
\begin{equation*}
    \tikz{
      \node[label=below:{$t_m$}] at (-4.5,0){ 
       \tikz[very thick,xscale=1.2]{
          \draw[dashed,thick] (-.7,-.5)-- (-.7,.5);
          \draw[dashed,thick] (1.7,-.5)-- (1.7,.5);
\draw (-.5,-.5)-- (-.5,.5);
          \draw (.5,-.5)-- (.5,.5) node [midway,tstar]{};
          \draw (1.5,-.5)-- (1.5,.5);
          \node at (1,0){$\cdots$};
          \node at (0,0){$\cdots$};
        }
      };
      \node[label=below:{$u_m$}] at (0,0){ 
        \tikz[very thick,xscale=1.2]{
          \draw[dashed,thick] (-.7,-.5)-- (-.7,.5);
          \draw[dashed,thick] (1.7,-.5)-- (1.7,.5);
          \draw (-.5,-.5)-- (-.5,.5);
          \draw (.5,-.5)-- (.5,.5) node [midway,fill=black,circle,inner
          sep=2pt]{};
          \draw (1.5,-.5)-- (1.5,.5);
          \node at (1,0){$\cdots$};
          \node at (0,0){$\cdots$};
        }
      };
      \node[label=below:{$(m,m+1)$}] at (4.5,0){ 
        \tikz[very thick,xscale=1.2]{
          \draw[dashed,thick] (-.7,-.5)-- (-.7,.5);
          \draw[dashed,thick] (1.7,-.5)-- (1.7,.5);
          \draw (-.5,-.5)-- (-.5,.5);
          \draw (.1,-.5)-- (.9,.5);
          \draw (.9,-.5)-- (.1,.5);
          \draw (1.5,-.5)-- (1.5,.5);
          \node at (1,0){$\cdots$};
          \node at (0,0){$\cdots$};
        }
      };
    }
  \end{equation*}
\begin{equation*}
    \tikz{
      \node[label=below:{$\sigma$}] at (-4.5,0){ 
       \tikz[very thick,xscale=1.2]{
          \draw[dashed,thick] (-.7,-.5)-- (-.7,.5);
          \draw[dashed,thick] (1.7,-.5)-- (1.7,.5);
\draw (-.5,-.5)-- (-.1,.5);
          \draw (.3,-.5)-- (.7,.5);
          \draw (1.1 ,-.5)-- (1.5,.5);
          \node at (.9,0){$\cdots$};
\draw (-.7,0)-- (-.5,.5);
           \draw (1.5 ,-.5)-- (1.7,0);
          \node at (.1,0){$\cdots$};
        }
      };
      \node[label=below:{$\tau$}] at (0,0){ 
        \tikz[very thick,xscale=1.2]{
          \draw[dashed,thick] (-.7,-.5)-- (-.7,.5);
          \draw[dashed,thick] (1.7,-.5)-- (1.7,.5);
 \draw (-.5,.5)-- (-.1,-.5);
          \draw (.3,.5)-- (.7,-.5);
          \draw (1.1 ,.5)-- (1.5,-.5);
          \node at (.9,0){$\cdots$};
\draw (-.7,0)-- (-.5,-.5);
           \draw (1.5 ,.5)-- (1.7,0);
          \node at (.1,0){$\cdots$};
        }
      };
    }
  \end{equation*}

The relations of $\Gamma$ and (\ref{degHeckeDO}--\ref{last}) are
determined by simple local rules such as:
\begin{equation*}
  \tikz[very thick,scale=1.5,baseline]{  
          \draw (.1,-.5)--node [pos=.25,ucircle]{} (.9,.5);
          \draw (.9,-.5)-- (.1,.5);
}-\tikz[very thick,scale=1.5,baseline]{  
          \draw(.1,-.5)--  node [pos=.75,ucircle] {}(.9,.5);
          \draw (.9,-.5)-- (.1,.5);
}=  \tikz[very thick,scale=1.5,baseline]{  
          \draw (.1,-.5)-- (.9,.5);
          \draw(.9,-.5)-- node [pos=.75,ucircle] {} (.1,.5);
} -\tikz[very thick,scale=1.5,baseline]{  
          \draw (.1,-.5)-- (.9,.5);
          \draw (.9,-.5)--node [pos=.25,ucircle]{} (.1,.5);
}=k \,\,\tikz[very thick,scale=1.5,baseline]{  
          \draw (.3,-.5)-- (.3,.5);
          \draw   (.7,-.5)-- (.7,.5);
}+k \,\,\tikz[very thick,scale=1.5,baseline]{  
          \draw (.3,-.5)-- node [midway,tstar]{}  (.3,.5);
          \draw   (.7,-.5)-- node [midway,tstar]{}  (.7,.5);
\node[scale=.7] at (.9,.2) {$-1$};
}+\cdots+k \,\,\tikz[very thick,scale=1.5,baseline]{  
          \draw (.3,-.5)-- node [midway,tstar]{}  (.3,.5);
          \draw   (.7,-.5)-- node [midway,tstar]{}  (.7,.5);
\node[scale=.7] at (.1,.2) {$-1$};}
\end{equation*}
\begin{equation*}
  \tikz[very thick,scale=1.5,baseline]{  
\draw[dashed,thick] (0,-.5)-- (0,.5);
\draw (-.4,-.5) --  node [pos=.75,tstar]{} (.4,.5);
}= \zeta\tikz[very thick,scale=1.5,baseline]{  
\draw[dashed,thick] (0,-.5)-- (0,.5);
\draw (-.4,-.5) --  node [pos=.25,tstar]{} (.4,.5);
} \qquad \qquad \qquad\qquad  \tikz[very thick,scale=1.5,baseline]{  
\draw[dashed,thick] (0,-.5)-- (0,.5);
\draw (.4,-.5) --  node [pos=.75,tstar]{} (-.4,.5);
}= \zeta^{-1}\tikz[very thick,scale=1.5,baseline]{  
\draw[dashed,thick] (0,-.5)-- (0,.5);
\draw (.4,-.5) --  node [pos=.25,tstar]{} (-.4,.5);
}\end{equation*}
\begin{equation*}
 \tikz[very thick,scale=1.5,baseline]{  
\draw[dashed,thick] (0,-.5)-- (0,.5);
\draw (-.4,-.5) --  node [pos=.75,ucircle]{} (.4,.5);
}= \tikz[very thick,scale=1.5,baseline]{  
\draw[dashed,thick] (0,-.5)-- (0,.5);
\draw (-.4,-.5) --  node [pos=.25,ucircle]{} (.4,.5);
}-\hbar \tikz[very thick,scale=1.5,baseline]{  
\draw[dashed,thick] (0,-.5)-- (0,.5);
\draw (-.4,-.5) --  (.4,.5);
}\qquad \qquad   \tikz[very thick,scale=1.5,baseline]{  
\draw[dashed,thick] (0,-.5)-- (0,.5);
\draw (.4,-.5) --  node [pos=.75,ucircle]{} (-.4,.5);
}= \tikz[very thick,scale=1.5,baseline]{  
\draw[dashed,thick] (0,-.5)-- (0,.5);
\draw (.4,-.5) --  node [pos=.25,ucircle]{} (-.4,.5);
}+\hbar \tikz[very thick,scale=1.5,baseline]{  
\draw[dashed,thick] (0,-.5)-- (0,.5);
\draw (.4,-.5) -- (-.4,.5);
}
\end{equation*}
\begin{equation*}
  \tikz[very thick,scale=1,baseline=-3pt]{  
\draw[dashed,thick] (0,-.75)-- (0,.75);
\draw  (-.4,-.75) to[out=90,in=-90] (.4,0) to[out=90,in=-90] (-.4,.75);
}= \tikz[very thick,scale=1,baseline=-3pt]{  
\draw[dashed,thick] (0,-.75)-- (0,.75);
\draw  (-.4,-.75)-- node [midway,ucircle]{} (-.4,.75);
}\,\,-p\Bigg(\,\, \tikz[very thick,scale=1,baseline=-3pt]{  
\draw[dashed,thick] (0,-.75)-- (0,.75);
\draw  (-.4,-.75)-- node [midway,tstar]{} (-.4,.75);
}\,\,\Bigg)\qquad \qquad   \tikz[very thick,scale=1,baseline=-3pt]{  
\draw[dashed,thick] (0,-.75)-- (0,.75);
\draw  (.4,-.75) to[out=90,in=-90] (-.4,0) to[out=90,in=-90] (.4,.75);
}= \tikz[very thick,scale=1,baseline=-3pt]{  
\draw[dashed,thick] (0,-.75)-- (0,.75);
\draw  (.4,-.75)-- node [midway,ucircle]{} (.4,.75);
}\,\, +\hbar\, \,\tikz[very thick,scale=1,baseline=-3pt]{  
\draw[dashed,thick] (0,-.75)-- (0,.75);
\draw  (.4,-.75)-- (.4,.75);
}\,\,-p\Bigg(\zeta^{-1}\tikz[very thick,scale=1,baseline=-3pt]{  
\draw[dashed,thick] (0,-.75)-- (0,.75);
\draw  (.4,-.75)-- node [midway,tstar]{} (.4,.75);
}\Bigg)
\end{equation*}

\nc{\cycle}{\chi}
\nc{\ncycle}{\upsilon}
Consider the permutations $\cycle_i=(i,i-1\dots, 1)$ and $\ncycle_i=(i,i+1,\dots, n)$.

\begin{theorem}\label{Cherednik-iso}
  The algebras $\alg$ and $\mathsf{H}$ are isomorphic via maps
  identifying the copies of $\bbK[\Gamma]$ and sending 
\begin{equation*}
  x_i\mapsto\cycle_i \sigma \ncycle_i^{-1} \qquad y_i\mapsto
  \ncycle_i \tau \cycle_i^{-1}\qquad \DO_i\mapsto u_i.
\end{equation*}  
\end{theorem}

The elements $\cycle_i \sigma \ncycle_i^{-1} $ and $\ncycle_i \tau
\cycle_i^{-1}$ have natural graphical representations:
\begin{equation*}
    \tikz{
      \node[label=below:{$ \cycle_i \sigma \ncycle_i^{-1} $}] at (-4.5,0){ 
       \tikz[very thick,xscale=1.2]{
          \draw[dashed,thick] (-.7,-.5)-- (-.7,.5);
          \draw[dashed,thick] (1.7,-.5)-- (1.7,.5);
\draw (-.5,-.5)-- (-.5,.5);
          \draw (.3,-.5)-- (.3,.5);
          \draw (.7,-.5)-- (.7,.5);
          \draw (1.5 ,-.5)-- (1.5,.5);
          \node at (1.1,.3){$\cdots$};
\draw (-.7,0) to[out=0,in=-90] (.5,.5);
           \draw (.5 ,-.5)to[out=90,in=180] (1.7,0);
          \node at (-.1,-.3){$\cdots$};
        }
      };
      \node[label=below:{$
  \ncycle_i \tau \cycle_i^{-1}$}] at (0,0){ 
        \tikz[very thick,xscale=1.2]{
          \draw[dashed,thick] (-.7,-.5)-- (-.7,.5);
          \draw[dashed,thick] (1.7,-.5)-- (1.7,.5);
 \draw (-.5,.5)-- (-.5,-.5);
          \draw (.7,.5)-- (.7,-.5);
          \draw (.3,.5)-- (.3,-.5);
          \draw (1.5 ,.5)-- (1.5,-.5);
          \node at (1.1,-.3){$\cdots$};
\draw (-.7,0)to[out=0,in=90] (.5,-.5);
           \draw (.5 ,.5) to[out=-90,in=180] (1.7,0);
          \node at (-.1,.3){$\cdots$};
        }
      };
    }
  \end{equation*}

\begin{proof}
First we need to check the compatibility of this map with the action
of $\Gamma$.  Note that the images of $x_i$ and $y_i$ commute with
transpositions except $(i,i\pm  1)$, and
\begin{align*}
  (i,i\pm 1)\cycle_i\sigma \ncycle_i^{-1} (i,i\pm 1) =   \cycle_{i\pm
                                                       1}\sigma
                                      \ncycle_{i\pm 1}^{-1}\qquad  (i,i\pm 1) \ncycle_i
                                                      \tau\cycle_i^{-1}(i,i\pm 1)=
                                                                            \ncycle_{i\pm
  1}\tau\cycle_{i\pm 1}^{-1}
\end{align*}
This establishes equivariance for $S_n\subset \Gamma$.  Furthermore, 
\begin{align*}
  \cycle_i\sigma \ncycle_i^{-1} t_i &=   \cycle_i\sigma t_n\ncycle_i^{-1}&   \ncycle_i\tau\cycle_i^{-1}t_i&= \ncycle_i\tau t_1
                                                                            \cycle_i^{-1}\\
&=   \cycle_i t_{n+1}\sigma 
                                      \ncycle_i^{-1}&&=
                                                                           \ncycle_i t_0\tau\cycle_i^{-1}\\
&=  t_{n+i} \cycle_i \sigma 
                                      \ncycle_i^{-1}&&=
                                                                           t_{i-n}\ncycle_i \tau\cycle_i^{-1}\\
&=  \zeta^{-1} t_{i} \cycle_i \sigma 
                                      \ncycle_i^{-1}&&=
                                                                           \zeta
                                                       t_{i}\ncycle_i \tau\cycle_i^{-1}.
\end{align*}
Similar calculations show that these elements commute with the other
$t_j$'s.  Thus, these elements have the correct commutation relations
with $\Gamma$ and we need only check that they have the correct
commutator with each other.

First, let us check that the images of $x_i$ and $x_j$ commute; we can assume that $j>i$.  Thus, we have that: 
\begin{align*}
\cycle_j \sigma \ncycle_j^{-1}\cycle_i \sigma \ncycle_i^{-1}&= \cycle_j \sigma \cycle_i \ncycle_j^{-1}\sigma \ncycle_i^{-1}\\
&= \cycle_j (i+1,i,\dots,2)\sigma^2  (j-1,j,\dots,n-1) \ncycle_i^{-1}\\
&= \cycle_i (j+1,i,\dots,2)\sigma^2  (i-1,j,\dots,n-1) \ncycle_j^{-1}\\
&= \cycle_i \sigma \ncycle_i^{-1}\cycle_j\sigma \ncycle_j^{-1}.
\end{align*}
This proof is perhaps easier to imagine using a picture:
\begin{equation*}
\tikz[very thick,xscale=1.2,baseline=10pt]{
          \draw[dashed,thick] (-.7,-.5)-- (-.7,1.5);
          \draw[dashed,thick] (1.7,-.5)-- (1.7,1.5);
\draw (-.5,-.5)-- (-.5,1.5);
          \draw (.5,-.5)-- (.5,1.5);
          \draw (.7,-.5)-- (.7,.5);
                    \draw (.3,.5)-- (.3,1.5);
          \draw (1.5 ,-.5)-- (1.5,1.5);
          \node at (1.1,.3){$\cdots$};
                    \node at (1.1,1.3){$\cdots$};
\draw (-.7,0) to[out=0,in=-90] (.3,.5);
           \draw (.3 ,-.5)to[out=90,in=180] (1.7,0);
           \draw (-.7,1) to[out=0,in=-90] (.7,1.5);
           \draw (.7 ,.5)to[out=90,in=180] (1.7,1);
          \node at (-.1,-.3){$\cdots$};
          \node at (-.1,.7){$\cdots$};}\,=\, \tikz[very thick,xscale=1.2,baseline=10pt]{
          \draw[dashed,thick] (-.7,-.5)-- (-.7,1.5);
          \draw[dashed,thick] (1.7,-.5)-- (1.7,1.5);
        \draw (-.5 ,-.5) to [in=-90,out=90] (-.5,0) to [in=-90,out=90] (-.1,.7) to[in=-90,out=90]  (-.5,1.5);          \draw (.5,-.5)-- (.5,1.5);         
          \draw (1.5 ,-.5) to [in=-90,out=90] (1.1,.3) to [in=-90,out=90] (1.5,1) to[in=-90,out=90]  (1.5,1.5);
                    \node at (1.1,1.3){$\cdots$};
\draw (-.7,.42) to[out=30,in=-150] (.3,1.5);
           \draw (.3 ,-.5)to[out=60,in=210] (1.7,.42);
                      \draw (.7 ,-.5)to[out=30,in=210] (1.7,.58);
           \draw (-.7,.58) to[out=30,in=-120] (.7,1.5);
          \node at (-.1,-.3){$\cdots$};} \,=\, \tikz[very thick,xscale=1.2,baseline=10pt]{
          \draw[dashed,thick] (-.7,-.5)-- (-.7,1.5);
          \draw[dashed,thick] (1.7,-.5)-- (1.7,1.5);
        \draw (-.5 ,-.5) to [in=-90,out=90] (-.5,0) to [in=-90,out=90] (-.1,.7) to[in=-90,out=90]  (-.5,1.5);          \draw (.5,-.5)-- (.5,1.5);
          \draw (1.5 ,-.5) to [in=-90,out=90] (1.1,.3) to [in=-90,out=90] (1.5,1) to[in=-90,out=90]  (1.5,1.5);
                    \node at (1.1,1.3){$\cdots$};
\draw (-.7,.58) to[out=30,in=-150] (.3,1.5);
           \draw (.3 ,-.5)to[out=60,in=210] (1.7,.58);
                      \draw (.7 ,-.5)to[out=30,in=210] (1.7,.42);
           \draw (-.7,.42) to[out=30,in=-120] (.7,1.5);
          \node at (-.1,-.3){$\cdots$};}\,=\,
          \tikz[very thick,xscale=1.2,baseline=10pt]{
          \draw[dashed,thick] (-.7,-.5)-- (-.7,1.5);
          \draw[dashed,thick] (1.7,-.5)-- (1.7,1.5);
\draw (-.5,-.5)-- (-.5,1.5);
          \draw (.5,-.5)-- (.5,1.5);
          \draw (.3,-.5)-- (.3,.5);
                    \draw (.7,.5)-- (.7,1.5);
          \draw (1.5 ,-.5)-- (1.5,1.5);
          \node at (1.1,.3){$\cdots$};
                    \node at (1.1,1.3){$\cdots$};
\draw (-.7,0) to[out=0,in=-90] (.7,.5);
           \draw (.7 ,-.5)to[out=90,in=180] (1.7,0);
           \draw (-.7,1) to[out=0,in=-90] (.3,1.5);
           \draw (.3 ,.5)to[out=90,in=180] (1.7,1);
          \node at (-.1,-.3){$\cdots$};
          \node at (-.1,.7){$\cdots$};}
\end{equation*}
The key relation is (\ref{isotope3}) which we use in the middle equality.

Next, we consider the commutation relation between $x_i$ and $y_i$, given in \eqref{xiyi}.  This we will prove in a few steps:
\begin{align}
  [\cycle_i \sigma \ncycle_i^{-1} ,\ncycle_i \tau
  \cycle_i^{-1}]&=\cycle_i \sigma \tau \cycle_i^{-1}-\ncycle_i \tau
                  \sigma \ncycle_i^{-1}\notag\\
&=\cycle_i (u_1-p(\zeta^{-1}t_1)+\hbar)\cycle_i^{-1}-\ncycle_i (u_n-p(t_n) )\ncycle_i^{-1} \label{eq:commutator1}\\
&=\hbar +p(t_i)-p(\zeta
  ^{-1}t_i)  +\cycle_i u_1 \cycle_i^{-1}-\ncycle_i u_n\ncycle_i
  ^{-1}.\notag
\end{align}
Note that 
\begin{multline}
p(t_i)-p(\zeta
  ^{-1}t_i) \\=
  \sum_{s=1}^{\ell-1}
  \sum_{r=0}^\ell\zeta^{-rs}h_rt_i^s-\sum_{s=1}^{\ell-1}
  \sum_{r=0}^\ell\zeta^{-(r+1)s}h_rt_i^s=\sum
  \sum\zeta^{-rs}(h_r-h_{r-1})t_i^s.\label{eq:commutator2}
\end{multline}

Similarly, 
\begin{align}
\cycle_i u_1 \cycle_i^{-1}&= u_i+k\sum_{j=1}^{i-1}\sum_{p=0}^{\ell-1} t_j^pt_i^{-p}(j,i); \\\ncycle_i u_n\ncycle_i ^{-1}&=
u_i-k\sum_{j=i+1}^{n}\sum_{p=0}^{\ell-1}  t_j^pt_i^{-p}(j,i) .\label{eq:commutator3}
\end{align}
Thus, combining (\ref{eq:commutator1}--\ref{eq:commutator3}), we can confirm \eqref{xiyi} as follows: 
\begin{equation}
  \label{eq:cherednik1}
    [\cycle_i \sigma \ncycle_i^{-1} ,\ncycle_i \tau
  \cycle_i^{-1}] =\hbar +\sum_{s=1}^{\ell-1}
  \sum_{r=0}^\ell\zeta^{-rs}(h_r-h_{r-1})t_i^s+k\sum_{i\neq j}\sum_{p=0}^{\ell-1}  t_j^pt_i^{-p}(j,i).
\end{equation}
Similarly, if $i\neq j$, then \begin{equation}
  \label{eq:cherednik2}
    [\cycle_i \sigma \ncycle_i^{-1} ,\ncycle_j \tau
  \cycle_j^{-1}] =\cycle_i \ncycle_j(\sigma (1,2) \tau -\tau (n-1,n)
  \sigma) \ncycle_i^{-1} \cycle_j^{-1}=-k \sum_{p=0}^{\ell-1}  \zeta^p t_i^pt_j^{-p}(i,j).
\end{equation} This confirms \eqref{xiyj}.

Thus, we have verified the existence of a map $\mathsf{H} \to \alg$.
Note that \[\DO_1=x_1y_1+p(\zeta^{-1}t_1)-\hbar\mapsto \sigma\tau+p(\zeta^{-1}t_1)-\hbar=u_1.\] By the  relations (\ref{degHeckeDO}) and (\ref{degHeckeDO}), this implies $\DO_i\mapsto u_i$ for all $i$.
  
The inverse is defined by 
\begin{equation}
    \sigma \mapsto (1,\dots, i)x_i(i,\dots,n)\qquad 
\tau \mapsto (n,\dots, i)y_i(i,\dots,1) \qquad u_i\mapsto \DO_i
\label{eq:sigma-image}
  \end{equation}
so this map is an isomorphism.
\end{proof}
\begin{lemma}\label{Euler}
  Under this isomorphism, the deformed Euler element $\mathsf{eu}$ of the
  Cherednik algebra matches $u_1+\cdots +u_n+n/2$.
\end{lemma}
\begin{proof}
  This follows immediately from the fact that $\DO_i\mapsto u_i$ and the formula for the deformed Euler element given in \cite[\S 2.3.5]{GoLo}.
\end{proof}

Thus, considering
the simultaneous eigenspaces of these operators gives a finer
decomposition of the Euler eigenspaces, which we will study in the
following section.
\begin{remark}
  The map of commutator with $\mathsf{eu}$ is semi-simple on
  $\mathsf{H}$, with all eigenvalues in $\Z$.  Thus, this conjugation
  induces a $\Z$-grading on $\mathsf{H}$, which is easy to describe in
  the presentation given above: the elements $\sigma$ and $\tau$ have
  degrees $1$ and $-1$, respectively, and all other generators have
  degree 0; we leave the verification of this based on the relations
  (\ref{degHeckeDO}--\ref{last}) to the reader.  Thus, in terms of
  diagrams, this grading measures the total winding number around the cylinder. 
\end{remark}

\begin{lemma}
  The elements $u_i,t_i$ for $i=1,\dots, n$ generate a subring $\mathsf{U}$
  of $A$ isomorphic
  to $\bbK[u_1,\dots, u_n,t_1,\dots, t_n]/(t_1^\ell-1, \dots, t_n^\ell-1)$.
\end{lemma}
\begin{proof}
  Obviously, the $t_i$ generate a copy of the the group ring on
  $(\Z/\ell\Z)^n$.  
  The elements $u_i$ commute by \eqref{u-commute}.  Furthermore, their images
  in the associated graded $\operatorname{gr} \mathsf{H}\cong \bbK[\Gamma]\otimes
  \bbK[\mathbf{x},\mathbf{y}]$ are given by $x_1y_1,\dots, x_ny_n$.
  Since these are algebraically independent over the group algebra, the $u_i$ are as well,
  and so they generate a copy of the polynomial ring.
\end{proof}

The subring $\mathsf{U}$ has another special property:
\begin{lemma}\label{lem:HC}
  The subalgebra $\mathsf{U}\subset \mathsf{H}$ is Harish-Chandra in the sense of \cite[\S
  1.3]{FOD}, that is, for any $a\in \mathsf{H}$, the bimodule
  $\mathsf{U}a \mathsf{U}$ is finitely generated as a left module or
  right module.  
\end{lemma}
\begin{proof}
  This is easily seen from the fact that for each fixed element of the
  affine Weyl group, the diagrams tracing out affine permutations with
  all possible decorations by dots and stars form a bimodule over
  $\mathsf{U}$ which is finitely generated as a left or right module.
  Of course, every $a\in \mathsf{H}$ lies in one of these submodules.
  This completes the proof, since $\mathsf{U}$ is Noetherian.  
\end{proof}

Note that this presentation allows us to give a ``strange'' polynomial
representation of the Cherednik algebra on the ring $\poly$ of
polynomials over $\bbK$ in the alphabets of variables $\mathbf{U}=\{U_1,\dots, U_n\}$ and
$\mathbf{T}=\{T_1,\dots, T_n\}$ modulo the relations $T_i^\ell=1$.  As before, we
define $U_i,T_i$ for all $i\in \Z$, by the formula
$U_i=U_{i-n}-\hbar, T_i=\zeta T_{i-n}$.  To distinguish between the
polynomial representation we wish to define and the action of $\Gamma$
on polynomials induced by its linear action, we use $f^{\sigma}$ to
denote the image of $f$ under the latter action of $\sigma\in \Gamma$.
The desired representation sends
\begin{align} u_i \cdot f(\mathbf{U};\mathbf{T}) &=U_if(\mathbf{U};\mathbf{T}) \label{eq:u-action} \\
t_i\cdot f(\mathbf{U};\mathbf{T})  & = T_i f(\mathbf{U};\mathbf{T}) \\
  r_i\cdot f (\mathbf{U};\mathbf{T}) &= f^{r_i}+k\ell \frac{
                   f^{(i,i+1)}-f}{U_{i+1}-U_i}\pi_{i,i+1}\\
  \label{eq:sigma-action}\sigma \cdot f(\mathbf{U};\mathbf{T}) & = (
                                                                 u_1-p(\zeta^{-1}t_1)+\hbar
                                                                 )\cdot\\
  & \qquad \quad f(U_2,U_3,\dots, U_n,U_{n+1};T_2,T_3,\dots, T_n,T_{n+1})\\
  \label{eq:tau-action}\tau \cdot f(\mathbf{U};\mathbf{T}) & = f(U_0,U_1,\dots, U_{n-2},U_{n-1};T_0,T_1,\dots, T_{n-2},T_{n-1})
\end{align}
where, as before, $\pi_{i,i+1}$ is the $\K[\mathbf{U}]$-linear map that sends
$T_1^{z_1}\cdots T_n^{z_n}\mapsto \delta_{z_i,z_{i+1}}T_1^{z_1}\cdots T_n^{z_n}.$ This is an extension to the
whole Cherednik algebra of the action by difference operators
introduced by Kodera-Nakajima in \cite[Thm. 1.5]{KoNa}.

This representation is generated by the constant function $1$, subject
to the left ideal of relations generated by 
\[(i,i+1)\cdot 1=\tau\cdot 1=1\qquad \sigma\cdot 1=( u_1-p(\zeta^{-1}t_1)+\hbar )\]
Note that if we transport structure from this representation to the
Cherednik algebra $\mathsf{H}$ then the formulae for the action of
$x_i$ and $y_i$ will be quite complicated.

Note also that the invariants of $\Gamma$ acting on the ring  $\poly$
are simply the $S_n$-invariant functions in the variables $U_i$ (for
the usual action or equivalently, the dAHA action).  Thus, the spherical Cherednik
algebra $e\mathsf{H}e$ acts naturally on these symmetric polynomials.

In the discussion above, we can think of the parameters as
formal variables, in which case, we'll obtain an action on $
\poly^{\Gamma}\otimes \param$, where $\param=\K[  \ps_0,\dots,
\ps_{\ell-1},k]^{S_{\ell}}$ where $\ps_i$ are as defined in \eqref{eq:s-def}.

\section{Weighted KLR algebras}
\label{sec:weight-klr-algebr}

This presentation gives a concrete equivalence between a category of
representations of the Cherednik algebra, and representations of a
weighted KLR algebra, originally proven in \cite{WebRou}.  In this
section, we set $\hbar=-1$ for simplicity\footnote{The reader may
  doubt the simplicity of this choice, but due to some other
  notational choices, it really is for the best.}, and assume that we have
numerical parameters $k,\ps_i\in \K$.
\begin{definition}
  Let $\mathsf{H}\mmod_u$ be the category of $\mathsf{H}$-modules on which the
  polynomial ring $\mathsf{U}$ acts locally finitely, with
  finite dimensional generalized weight spaces.  We call modules in
  this category {\bf Dunkl-Opdam modules.}  In the terminology of
  \cite{FOD}, these are the ``Harish-Chandra'' modules for this subalgebra.
\end{definition}
By Lemma \ref{Euler}, any module where the Euler element $\mathsf{eu}$
acts with finite dimensional generalized weight spaces lies in this
category.  In particular, any module in the GGOR category $\cO$ is a
Dunkl-Opdam module.   

Of course, for each pair $\Ba\in\K^n$ and $\Bz\in \mu_\ell(\K)^n$, we
have an exact generalized weight space functor \[W_{\Ba,\Bz}(M)=\{m\in M\mid
(u_i-a_i)^Nm=(t_i-z_i)^Nm=0 \text{ for } N\gg 0\}.\] 

Consider the additive quotient group $\K/\Z$; for an element
$a\in \K$, we let $\bar{a}$ denote its coset in this quotient. We have a
natural homomorphism $\gamma\colon \mu_\ell\to \K/ \Z$ sending $\zeta^m\mapsto
\frac{m}{\ell}\pmod{\Z}$.  Let $\Sigma\colon \K\times \mu_\ell(\K)\to \K/
\Z$ be the homomorphism $\Sigma(a,z)=\frac{\bar{a}}{\ell}+\gamma(z)$.
Note that this is well-defined since the characteristic of $\K$ is
coprime to $\ell$.

Consider the length 0 element
\begin{align*}
\nu\cdot (\Ba,\Bz)&=((a_0=a_n+1,a_1,\dots,
a_{n-1}),(z_0=\zeta z_{n},\dots,z_{n-1}))\\ \nu^{-1}\cdot (\Ba,\Bz)&=((a_2,a_3,\dots,
a_{n+1}=a_1-1),(z_2,\dots,z_{n+1}=\zeta^{-1} z_1)).
\end{align*}

The relations (\ref{u-sig}--\ref{t-tau}) show that:
\begin{lemma}\label{lem:st-transform}
  The elements $\sigma$ and $\tau$ induce natural transformations:
\[  \sigma\colon W_{(\Ba,\Bz)}\to W_{\nu\cdot (\Ba,\Bz)}\qquad \tau\colon W_{(\Ba,\Bz)}\to W_{\nu^{-1}\cdot (\Ba,\Bz)}.\]
\end{lemma}

\begin{lemma}
Let $v\in W_{\Ba,\Bz}(M)$ be a weight vector that generates $M$. If
for some $\Ba',\Bz'$ we have $W_{\Ba',\Bz'}(M)\neq 0$ then after some permutation
$\rho\in S_n$, we have that $\Sigma (a_i,z_i)=\Sigma(a_{\rho \cdot
  i}',z_{\rho\cdot i}')$.   
\end{lemma}
\begin{proof}
  This is readily confirmed from the relations
 (\ref{u-sig}--\ref{t-tau}).  The (\ref{degHeckeDO}) shows that the
  action of $\Gamma$ can only simultaneously permute $\Ba$ and $\Bz$,
  and Lemma \ref{lem:st-transform} shows that $\Sigma$ and $\tau$ act by
  simultaneous cyclic permutation of $\Sigma (a_i,z_i)$.  
\end{proof}

\begin{corollary}
  If $M$ is an indecomposable $\mathsf{H}$-module, 
  and we have $(\Ba,\Bz)$ and $(\Ba',\Bz')$ such that
  $W_{\Ba,\Bz}(M)\neq 0$ and $W_{\Ba',\Bz'}(M)\neq 0$  then the multisets 
  $\{\Sigma (a_i,z_i)\}$ and   $\{\Sigma (a_i',z_i')\}$ are equal.
\end{corollary}
In particular, we can naturally organize the structure of modules over
$\mathsf{H}$ by fixing which elements of $\K/\Z$ can appear as
$\Sigma(a_i,z_i)$.   Fix a subset $D$ of $\K/\Z$, and let
$\tilde{D}=\Sigma^{-1}(D). $
\begin{definition}\label{def:HmodD}
  Let $\mathsf{H}\mmod_D$ be the subcategory of $\mathsf{H}\mmod_u$
  killed by the functors $W_{\Ba,\Bz}$ where
  $(a_i,z_i)\notin \tilde{D}$ for some $i$.
\end{definition}
 We'll see that the structure of this
category depends in a subtle way on the set $D$; we'll need a fair
amount of combinatorics below to capture this structure.  The most
important aspect of it a quiver structure on $D$ that we'll define below.
We give
$D$ the structure of a quiver by adding an arrow $m\to m+k $ whenever
both lie in $D$.  
Thus if $\K$ is a field of characteristic 0, if $k =a/e\in \Q$, then
$\K/\Z$ is an infinite union of $e$-cycles, whereas if $k\in
\K\setminus \Q$ then $\K/\Z$ is a union of infinite linear quivers.

\subsection{Characteristic $0$}
\label{sec:characteristic-0}

Assume that $\K$ is a field of characteristic 0, and thus contains a
canonically embedded copy of $\Q$.
Accordingly, $\K$ is a $\Q$-vector space,
and using the axiom of choice, we can choose a $\Q$-linear map
$\Upsilon\colon \K\to
\R$ which sends $1\mapsto 1$.  
Note that making this choice, we have a  divergence between two
important cases: if $k \in \Q$, then we must have $\Upsilon(k )=k $; on
the other hand, if $k \notin \Q$, then $\Upsilon(k )$ can be chosen
freely.  For example, in the latter case, we could without loss of
generality assume that $\Upsilon(k )=0$.  Note that while our precise
description of the attached weighted KLR algebra will depend on the
choice of $\Upsilon$, this choice is purely auxilliary, and changing
it will result in two algebras which are isomorphic by \cite[2.15]{WebwKLR}.

In this case, if $k =a/e\in \Q$, then $\K/\Z$ is an infinite union of $e$-cycles, whereas if $k\in \K\setminus \Q$ then $\K/\Z$ is a union of infinite linear quivers.
Let $\cs_i=\overline{\ps_i/\ell}\in \K/\Z$.  

The category $\mathsf{H}\mmod_D$ from Definition \ref{def:HmodD} has a natural description in terms of weighted KLR
algebras.
\begin{definition}
  Let $D_\infty$ be the quiver $D$ with an additional vertex $\infty$
  added, and an arrow $\infty \to \cs_i$ added
  for each $i$ such that $\cs_i\in D$.  This is what
  we often call a {\bf Crawley-Boevey quiver}, after the observation
  by Crawley-Boevey that the points in Nakajima's quiver varieties can
  be seen as representations of the doubling of this quiver, with a
  1-dimensional vector space at $\infty$.
\end{definition}
Choose a real number $\epsilon$ such that $0<\epsilon \ll 1$.  
Consider the weighting of this quiver where each
edge in $D$ is weighted by $\Upsilon(k )$
and the new edge for $\cs_i$ by
$\Upsilon(\frac{p(\zeta^i)}{\ell} )-i\epsilon$. Note that this means that two new edges
connected to the same vertex can never have the same weighting, since
if $\Upsilon(\frac{p(\zeta^i)}{\ell}
)=\Upsilon(\frac{p(\zeta^j)}{\ell} )$, then
$\Upsilon(\cs_i-\cs_j)=\frac{i-j}{\ell}\notin \Z$.  

\begin{example}
  For example, if $\ell=2, k=2/3$, and $s_0=0, s_1=1/3$ then we have that $\K/\Z$
  breaks into 3-cycles 
\[\bar a \to \bar a+\frac{2}3\to \bar
  a+\frac{4}3 \to \bar {a}+2=\bar{a}.\]
If $D=\{\bar{0},\overline{1/3},\overline{2/3}\}$, then the Crawley-Boevey quiver
is given by this 3-cycle with edges from $\bar{0}$ and $\overline{1/3}$ to $\infty$.
On the other hand, if $D$ is disjoint from
$\{\bar{0},\overline{1/3},\overline{2/3}\}$, then the Crawley-Boevey quiver adds
no edges.

On the other hand, if $k=\sqrt{2}$ (assuming this root exists in
$\K$), then $\K/\Z$ will decompose into infinite chains 
$\cdots \to  \bar a -\sqrt{2}\to \bar a \to \bar a+\sqrt{2}\to
\cdots$.  Note that $k$ being an irrational algebraic number has no
bearing on the structure of the category; the only thing which is
significant is its order as an element of the group $\K/\Z$.  Since
$\Upsilon(k)=0$, this graph has trivial weighting.

The extra edges in the Crawley-Boevey quiver still attached to
$\bar{0}$ and $\overline{1/3}$, but these are now on different
components.  
\end{example}

\subsection{Weighted KLR algebras}
\label{sec:weight-klr-algebr-1}

Consider the {\bf reduced weighted KLR algebra} $R_D$ attached to the quiver
$D_\infty$ with its chosen weighting as defined in \cite[\S
4.1]{WebRou} (see also \cite[\S
3.1]{WebwKLR}).  
Choose $\epsilon \in \R$ to be smaller than $|\Upsilon(a_i-a_j)|/n$
for any pair $i$ and $j$ with $\Upsilon(a_i-a_j)\neq 0$.

\begin{definition}
  We let a  {\bf weighted
KLR diagram} be a collection of curves in
  $\R\times [0,1]$ with each curve mapping diffeomorphically to
  $[0,1]$ via the projection to the $y$-axis.  Each curve is allowed
  to carry any number of dots, and has a label that lies in $D$. We draw:
  \begin{itemize}
  \item a dashed line $\Upsilon(k )$ units to the right of each strand,
    which we call a {\bf ghost},
\item red
lines at $x=\Upsilon(\frac{p(\zeta^i)}{\ell} )-i\epsilon$  labeled
with the fundamental weight for $\cs_i\in D$.  
  \end{itemize} 
We now require that there are no
  triple points or tangencies involving any combination of strands,
  ghosts or red lines and no dots lie on crossings.  We consider these diagrams equivalent if they
  are related by an isotopy that avoids these tangencies, double
  points and dots on crossings.
\end{definition} 
The intersection of such a diagram with $y=0$ or $y=1$ gives a {\bf
  loading}, that is, a labeling of a finite subset of $\R$
with vertices of the quiver $D$.   
  For every pair of $n$-tuples $\Ba$ and $\Bz$ with
  $\Sigma(a_i,z_i)\in D$ for all $i$, we can define
  a loading $e(\Ba,\Bz)$ as follows: we label the real number
  $\Upsilon(\frac{a_i}{\ell})+i\epsilon$ with the element $\Sigma(a_i,z_i)\in D$.
\begin{definition}
  Consider the algebra $R_D$ spanned by weighted KLR diagrams whose
  top and bottom both give loadings of the form $e(\Ba,\Bz)$ with
  $\Sigma(a_i,z_i)\in D$ modulo the local relations

\newseq
\begin{equation*}\subeqn\label{dots-1}
    \begin{tikzpicture}[scale=.45,baseline]
      \draw[very thick](-4,0) +(-1,-1) -- +(1,1) node[below,at start]
      {$i$}; \draw[very thick](-4,0) +(1,-1) -- +(-1,1) node[below,at
      start] {$j$}; \fill (-4.5,.5) circle (5pt);
      % \draw[very thick] (0,0) +(0,-1) -- +(0,1) node[below, at
      % start]{$i$}; \fill (0,0) circle (5pt);
      \node at (-2,0){=}; \draw[very thick](0,0) +(-1,-1) -- +(1,1)
      node[below,at start] {$i$}; \draw[very thick](0,0) +(1,-1) --
      +(-1,1) node[below,at start] {$j$}; \fill (.5,-.5) circle (5pt);
      \node at (4,0){for $i\neq j$};
    \end{tikzpicture}\end{equation*}
\begin{equation*}\label{dots-2}\subeqn
    \begin{tikzpicture}[scale=.45,baseline]
      \draw[very thick](-4,0) +(-1,-1) -- +(1,1) node[below,at start]
      {$i$}; \draw[very thick](-4,0) +(1,-1) -- +(-1,1) node[below,at
      start] {$i$}; \fill (-4.5,.5) circle (5pt);
      % \draw[very thick] (0,0) +(0,-1) -- +(0,1) node[below, at
      % start]{$i$}; \fill (0,0) circle (5pt);
      \node at (-2,0){=}; \draw[very thick](0,0) +(-1,-1) -- +(1,1)
      node[below,at start] {$i$}; \draw[very thick](0,0) +(1,-1) --
      +(-1,1) node[below,at start] {$i$}; \fill (.5,-.5) circle (5pt);
      \node at (2,0){+}; \draw[very thick](4,0) +(-1,-1) -- +(-1,1)
      node[below,at start] {$i$}; \draw[very thick](4,0) +(0,-1) --
      +(0,1) node[below,at start] {$i$};
    \end{tikzpicture}\qquad 
    \begin{tikzpicture}[scale=.45,baseline]
      \draw[very thick](-4,0) +(-1,-1) -- +(1,1) node[below,at start]
      {$i$}; \draw[very thick](-4,0) +(1,-1) -- +(-1,1) node[below,at
      start] {$i$}; \fill (-4.5,-.5) circle (5pt);
      % \draw[very thick] (0,0) +(0,-1) -- +(0,1) node[below, at
      % start]{$i$}; \fill (0,0) circle (5pt);
      \node at (-2,0){=}; \draw[very thick](0,0) +(-1,-1) -- +(1,1)
      node[below,at start] {$i$}; \draw[very thick](0,0) +(1,-1) --
      +(-1,1) node[below,at start] {$i$}; \fill (.5,.5) circle (5pt);
      \node at (2,0){+}; \draw[very thick](4,0) +(-1,-1) -- +(-1,1)
      node[below,at start] {$i$}; \draw[very thick](4,0) +(0,-1) --
      +(0,1) node[below,at start] {$i$};
    \end{tikzpicture}
  \end{equation*}
\begin{equation*}\label{strand-bigon}\subeqn
    \begin{tikzpicture}[very thick,scale=.8,baseline]
      \draw (-2.8,0) +(0,-1) .. controls (-1.2,0) ..  +(0,1)
      node[below,at start]{$i$}; \draw (-1.2,0) +(0,-1) .. controls
      (-2.8,0) ..  +(0,1) node[below,at start]{$i$}; \node at (-.5,0)
      {=}; \node at (0.4,0) {$0$};
\node at (1.5,.05) {and};
    \end{tikzpicture}
\hspace{.4cm}
    \begin{tikzpicture}[very thick,scale=.8 ,baseline]

      \draw (-2.8,0) +(0,-1) .. controls (-1.2,0) ..  +(0,1)
      node[below,at start]{$i$}; \draw (-1.2,0) +(0,-1) .. controls
      (-2.8,0) ..  +(0,1) node[below,at start]{$j$}; \node at (-.5,0)
      {=};

\draw (1.8,0) +(0,-1) -- +(0,1) node[below,at start]{$j$};
      \draw (1,0) +(0,-1) -- +(0,1) node[below,at start]{$i$}; 
    \end{tikzpicture}
  \end{equation*} 
\begin{equation*}\label{ghost-bigon1}\subeqn
\begin{tikzpicture}[very thick,xscale=1.6 ,yscale=.8,baseline]
 \draw (1,-1) to[in=-90,out=90]  node[below, at start]{$i$} (1.5,0) to[in=-90,out=90] (1,1)
;
  \draw[dashed] (1.5,-1) to[in=-90,out=90] (1,0) to[in=-90,out=90] (1.5,1);
  \draw (2.5,-1) to[in=-90,out=90]  node[below, at start]{$j$} (2,0) to[in=-90,out=90] (2.5,1);
\node at (3,0) {=};
  \draw (3.7,-1) -- (3.7,1) node[below, at start]{$i$}
 ;
  \draw[dashed] (4.2,-1) to (4.2,1);
  \draw (5.2,-1) -- (5.2,1) node[below, at start]{$j$};  \node at (6.2,0){for $i+k\neq j$};
\end{tikzpicture}
\end{equation*} 
\begin{equation*}\label{ghost-bigon1a}\subeqn
\begin{tikzpicture}[very thick,xscale=1.6 ,yscale=.8,baseline]
 \draw (1.5,-1) to[in=-90,out=90]  node[below, at start]{$i$} (1,0) to[in=-90,out=90] (1.5,1)
;
  \draw[dashed] (1,-1) to[in=-90,out=90] (1.5,0) to[in=-90,out=90] (1,1);
  \draw (2,-1) to[in=-90,out=90]  node[below, at start]{$j$} (2.5,0) to[in=-90,out=90] (2,1);
\node at (3,0) {=};
  \draw (4.2,-1) -- (4.2,1) node[below, at start]{$i$}
 ;
  \draw[dashed] (3.7,-1) to (3.7,1);
  \draw (4.7,-1) -- (4.7,1) node[below, at start]{$j$};  \node at (6.2,0){for $i+k\neq j$};
\end{tikzpicture}
\end{equation*} \begin{equation*}\label{ghost-bigon2}\subeqn
\begin{tikzpicture}[very thick,xscale=1.4,baseline=25pt]
 \draw (1,0) to[in=-90,out=90]  node[below, at start]{$i$} (1.5,1) to[in=-90,out=90] (1,2)
;
  \draw[dashed] (1.5,0) to[in=-90,out=90] (1,1) to[in=-90,out=90] (1.5,2);
  \draw (2.5,0) to[in=-90,out=90]  node[below, at start]{$i+k$} (2,1) to[in=-90,out=90] (2.5,2);
\node at (3,1) {=};
  \draw (3.7,0) -- (3.7,2) node[below, at start]{$i$}
 ;
  \draw[dashed] (4.2,0) to (4.2,2);
  \draw (5.2,0) -- (5.2,2) node[below, at start]{$i+k$} node[midway,fill,inner sep=2.5pt,circle]{};
\node at (5.75,1) {$-$};

  \draw (6.2,0) -- (6.2,2) node[below, at start]{$i$} node[midway,fill,inner sep=2.5pt,circle]{};
  \draw[dashed] (6.7,0)-- (6.7,2);
  \draw (7.7,0) -- (7.7,2) node[below, at start]{$i+k$};
\node at (8.15,1) {$+$}; \node at (8.45,1){$h$};
  \draw (8.7,0) -- (8.7,2) node[below, at start]{$i$};
  \draw[dashed] (9.2,0)-- (9.2,2);
  \draw (10.2,0) -- (10.2,2) node[below, at start]{$i+k$};
\end{tikzpicture}
\end{equation*}
 \begin{equation*}\label{ghost-bigon2a}\subeqn
\begin{tikzpicture}[very thick,xscale=1.4,baseline=25pt]
 \draw (1.5,0) to[in=-90,out=90]  node[below, at start]{$i$} (1,1) to[in=-90,out=90] (1.5,2)
;
  \draw[dashed] (1,0) to[in=-90,out=90] (1.5,1) to[in=-90,out=90] (1,2);
  \draw (2,0) to[in=-90,out=90]  node[below, at start]{$i+k$} (2.5,1) to[in=-90,out=90] (2,2);
\node at (3,1) {=};
  \draw (4.2,0) -- (4.2,2) node[below, at start]{$i$}
 ;
  \draw[dashed] (3.7,0) to (3.7,2);
  \draw (4.7,0) -- (4.7,2) node[below, at start]{$i+k$} node[midway,fill,inner sep=2.5pt,circle]{};
\node at (5.25,1) {$-$};

  \draw (6.2,0) -- (6.2,2) node[below, at start]{$i$} node[midway,fill,inner sep=2.5pt,circle]{};
  \draw[dashed] (5.7,0)-- (5.7,2);
  \draw (6.7,0) -- (6.7,2) node[below, at start]{$i+k$};
\node at (7.15,1) {$+$}; \node at (7.45,1){$h$};
  \draw (8.2,0) -- (8.2,2) node[below, at start]{$i$};
  \draw[dashed] (7.7,0)-- (7.7,2);
  \draw (8.7,0) -- (8.7,2) node[below, at start]{$i+k$};
\end{tikzpicture}
\end{equation*}
 \begin{equation*}\subeqn\label{triple-boring}
    \begin{tikzpicture}[very thick,scale=1 ,baseline]
      \draw (-3,0) +(1,-1) -- +(-1,1) node[below,at start]{$m$}; \draw
      (-3,0) +(-1,-1) -- +(1,1) node[below,at start]{$i$}; \draw
      (-3,0) +(0,-1) .. controls (-4,0) ..  +(0,1) node[below,at
      start]{$j$}; \node at (-1,0) {=}; \draw (1,0) +(1,-1) -- +(-1,1)
      node[below,at start]{$m$}; \draw (1,0) +(-1,-1) -- +(1,1)
      node[below,at start]{$i$}; \draw (1,0) +(0,-1) .. controls
      (2,0) ..  +(0,1) node[below,at start]{$j$};
    \end{tikzpicture}
  \end{equation*}
\begin{equation*}\subeqn \label{eq:triple-point1}
    \begin{tikzpicture}[very thick,xscale=1.6,yscale=.8,baseline]
      \draw[dashed] (-3,0) +(.4,-1) -- +(-.4,1);
 \draw[dashed]      (-3,0) +(-.4,-1) -- +(.4,1); 
    \draw (-1.5,0) +(.4,-1) -- +(-.4,1) node[below,at start]{$i+k$}; \draw
      (-1.5,0) +(-.4,-1) -- +(.4,1) node[below,at start]{$i+k$}; 
 \draw (-3,0) +(0,-1) .. controls (-3.5,0) ..  +(0,1) node[below,at
      start]{$i$};\node at (-.75,0) {=};  \draw[dashed] (0,0) +(.4,-1) -- +(-.4,1);
 \draw[dashed]      (0,0) +(-.4,-1) -- +(.4,1); 
    \draw (1.5,0) +(.4,-1) -- +(-.4,1) node[below,at start]{$i+k$}; \draw
      (1.5,0) +(-.4,-1) -- +(.4,1) node[below,at start]{$i+k$}; 
 \draw (0,0) +(0,-1) .. controls (.5,0) ..  +(0,1) node[below,at
      start]{$i$};
\node at (2.25,0)
      {$-$};   
     \draw (4.5,0)
      +(.4,-1) -- +(.4,1) node[below,at start]{$i+k$}; \draw (4.5,0)
      +(-.4,-1) -- +(-.4,1) node[below,at start]{$i+k$}; 
 \draw[dashed] (3,0)
      +(.4,-1) -- +(.4,1); \draw[dashed] (3,0)
      +(-.4,-1) -- +(-.4,1); 
\draw (3,0)
      +(0,-1) -- +(0,1) node[below,at start]{$i$};
%\node[inner ysep=8pt,inner xsep=5pt,fill=white,draw,scale=.8] at (6.2,0){$\displaystyle \frac{Q_{ij}(y_3,y_2)-Q_{ij}(y_1,y_2)}{y_3-y_1}$};
    \end{tikzpicture}
  \end{equation*}
\begin{equation*}\subeqn\label{eq:triple-point2}
    \begin{tikzpicture}[very thick,xscale=1.6,yscale=.8,baseline]
\draw[dashed] (-3,0) +(0,-1) .. controls (-3.5,0) ..  +(0,1) ;  
  \draw (-3,0) +(.4,-1) -- +(-.4,1) node[below,at start]{$i$}; \draw
      (-3,0) +(-.4,-1) -- +(.4,1) node[below,at start]{$i$}; 
 \draw (-1.5,0) +(0,-1) .. controls (-2,0) ..  +(0,1) node[below,at
      start]{$i+k$};\node at (-.75,0) {=};  
    \draw (0,0) +(.4,-1) -- +(-.4,1) node[below,at start]{$i$}; \draw
      (0,0) +(-.4,-1) -- +(.4,1) node[below,at start]{$i$}; 
 \draw[dashed] (0,0) +(0,-1) .. controls (.5,0) ..  +(0,1);
 \draw (1.5,0) +(0,-1) .. controls (2,0) ..  +(0,1) node[below,at
      start]{$i+k$};
\node at (2.25,0)
      {$+$};   
     \draw (3,0)
      +(.4,-1) -- +(.4,1) node[below,at start]{$i$}; \draw (3,0)
      +(-.4,-1) -- +(-.4,1) node[below,at start]{$i$}; 
\draw[dashed] (3,0)
      +(0,-1) -- +(0,1);\draw (4.5,0)
      +(0,-1) -- +(0,1) node[below,at start]{$i+k$};
%\node[inner ysep=8pt,inner xsep=5pt,fill=white,draw,scale=.8] at (6.2,0){$\displaystyle \frac{Q_{ij}(y_3,y_2)-Q_{ij}(y_1,y_2)}{y_3-y_1}$};
    \end{tikzpicture}.
  \end{equation*}
  \begin{equation*}\subeqn\label{cost}
  \begin{tikzpicture}[very thick,baseline,scale=.8,xscale=.7]
    \draw (-2.8,0)  +(0,-1) .. controls (-1.2,0) ..  +(0,1) node[below,at start]{$i$};
       \draw[wei] (-1.2,0)  +(0,-1) .. controls (-2.8,0) ..  +(0,1) node[below,at start]{$i$};
           \node at (-.3,0) {=};
    \draw[wei] (2.8,0)  +(0,-1) -- +(0,1) node[below,at start]{$i$};
       \draw (1.2,0)  +(0,-1) -- +(0,1) node[below,at start]{$i$};
       \fill (1.2,0) circle (3pt);
 \node at (3.8,0) {$-$};\node at (4.5,0){$z_k$};
        \draw[wei] (6.8,0)  +(0,-1) -- +(0,1) node[below,at start]{$i$};
       \draw (5.2,0)  +(0,-1) -- +(0,1) node[below,at start]{$i$};
  \end{tikzpicture}\qquad\qquad
  \begin{tikzpicture}[very thick,baseline,scale=.8,xscale=.7]
          \draw[wei] (6.8,0)  +(0,-1) .. controls (5.2,0) ..  +(0,1) node[below,at start]{$j$};
  \draw (5.2,0)  +(0,-1) .. controls (6.8,0) ..  +(0,1) node[below,at start]{$i$};
           \node at (7.7,0) {=};
    \draw (9.2,0)  +(0,-1) -- +(0,1) node[below,at start]{$i$};
       \draw[wei] (10.8,0)  +(0,-1) -- +(0,1) node[below,at start]{$j$};
  \end{tikzpicture}
\end{equation*} 
  \begin{equation*}\subeqn
    \begin{tikzpicture}[very thick,scale=.8,baseline]
      \draw (-3,0)  +(1,-1) -- +(-1,1) node[at start,below]{$i$};
      \draw (-3,0) +(-1,-1) -- +(1,1)node [at start,below]{$j$};
      \draw[wei] (-3,0)  +(0,-1) .. controls (-4,0) ..  node[at start,below]{$m$}+(0,1);
      \node at (-1,0) {=};
      \draw (1,0)  +(1,-1) -- +(-1,1) node[at start,below]{$i$};
      \draw (1,0) +(-1,-1) -- +(1,1) node [at start,below]{$j$};
      \draw[wei] (1,0) +(0,-1) .. controls (2,0) .. node[at start,below]{$m$} +(0,1);   
\node at (2.8,0) {$+ $};
      \draw (6.5,0)  +(1,-1) -- +(1,1)  node[at start,below]{$i$};
      \draw (6.5,0) +(-1,-1) -- +(-1,1) node [at start,below]{$j$};
      \draw[wei] (6.5,0) +(0,-1) -- node[at start,below]{$m$} +(0,1);
\node at (3.8,-.2){$\delta_{i,j,m} $}  ;
 \end{tikzpicture}
  \end{equation*}
\begin{equation*}\subeqn\label{dumb}
    \begin{tikzpicture}[very thick,scale=.8,baseline]
      \draw[wei] (-3,0)  +(1,-1) -- +(-1,1);
      \draw (-3,0)  +(0,-1) .. controls (-4,0) ..  +(0,1);
      \draw (-3,0) +(-1,-1) -- +(1,1);
      \node at (-1,0) {=};
      \draw[wei] (1,0)  +(1,-1) -- +(-1,1);
  \draw (1,0)  +(0,-1) .. controls (2,0) ..  +(0,1);
      \draw (1,0) +(-1,-1) -- +(1,1);    \end{tikzpicture}\qquad \qquad
    \begin{tikzpicture}[very thick,scale=.8,baseline]
  \draw(-3,0) +(-1,-1) -- +(1,1);
  \draw[wei](-3,0) +(1,-1) -- +(-1,1);
\fill (-3.5,-.5) circle (3pt);
\node at (-1,0) {=};
 \draw(1,0) +(-1,-1) -- +(1,1);
  \draw[wei](1,0) +(1,-1) -- +(-1,1);
\fill (1.5,.5) circle (3pt);
    \end{tikzpicture}
  \end{equation*}
For the relations (\ref{dumb}), we also include their mirror images.

This algebra is graded with \[
\deg\tikz[baseline,very thick,scale=1.5]{\draw (.2,.3) --
  (-.2,-.1) node[at end,below, scale=.8]{$i$}; \draw
  (.2,-.1) -- (-.2,.3) node[at start,below,scale=.8]{$j$};} =-2\delta_{i,j} \qquad  \deg\tikz[baseline,very thick,scale=1.5]{\draw[densely dashed] 
  (-.2,-.1)-- (.2,.3) node[at start,below, scale=.8]{$i$}; \draw
  (.2,-.1) -- (-.2,.3) node[at start,below,scale=.8]{$j$};} =\delta_{j,i-k} \qquad \deg\tikz[baseline,very thick,scale=1.5]{\draw (.2,.3) --
  (-.2,-.1) node[at end,below, scale=.8]{$i$}; \draw [densely dashed]
  (.2,-.1) -- (-.2,.3) node[at start,below,scale=.8]{$j$};}
=\delta_{j,i+k}\]
\[
\deg\tikz[baseline,very thick,scale=1.5]{\draw
  (0,.3) -- (0,-.1) node[at end,below,scale=.8]{$i$}
  node[midway,circle,fill=black,inner
  sep=2pt]{};}=2 \qquad \deg\tikz[baseline,very thick,scale=1.5]{\draw[wei] 
  (-.2,-.1)-- (.2,.3) node[at start,below, scale=.8]{$i$}; \draw
  (.2,-.1) -- (-.2,.3) node[at start,below,scale=.8]{$j$};} =\delta_{j,i} \qquad \deg\tikz[baseline,very thick,scale=1.5]{\draw (.2,.3) --
  (-.2,-.1) node[at end,below, scale=.8]{$i$}; \draw [wei]
  (.2,-.1) -- (-.2,.3) node[at start,below,scale=.8]{$j$};}
=\delta_{j,i}\] and we'll also consider the completion $\widehat{R}_D$ of this
  algebra with respect to its grading.  
\end{definition}

We let $e(\Ba,\Bz)$ denote the
  idempotent in $R_D$ or the completion $\widehat{R}_D$ given by a
  diagram of vertical lines whose $x$-values are determined by the
  corresponding loading.

  It will often be technically more convenient for us to think of
  $R_D$ or $\widehat{R}_D$ as a category whose objects are loadings
  and whose morphisms are elements of $R_D$ matching the source
  loading at the bottom and target loading at the top; this is the
  standard trick for considering a ring with set of idempotents
  summing to the identity as a category, discussed in \cite[\S 3.1]{Lauintro}.

  \begin{remark}\label{rem:equiv}
    As we've defined it, the algebra $R_D$ is infinite rank as a
    module over $\bbK[y_1,\dots, y_n]$, since we consider the
    $x$-values of the strands at the top and bottom of the diagram as
    fixed.  However, if two loadings are related by an isotopy
    (i.e. the straight line diagram relating them has no crossings),
    they are equivalent objects in the category $R_D$.  This is {\bf
      equivalence} of loadings, as discussed in
    \cite[Def. 2.9]{WebwKLR}.  As in \cite[Def. 2.13]{WebwKLR}, we
    usually take ``weighted KLR algebra'' to mean the algebra Morita
    equivalent to $R_D$ where we keep only one loading from each
    equivalence class.  
  \end{remark}
  
\subsection{The isomorphism}
\label{sec:isomorphism}

We'll now compare this KLR algebra with the category of Dunkl-Opdam
modules using the approach of Drozd-Futorny-Ovsienko \cite{FOD}.  They
introduce a category $\mathcal{H}$ whose objects are pairs
$(\Ba,\Bz)\in \K^n\times \mu_\ell(\K)^n$, considered as maximal ideals
$\mathfrak{m}_{(\Ba,\Bz)}\subset \mathsf{U}$.  The morphisms in this
category are given by:
\[\Hom_{\mathcal{H}} ((\Ba,\Bz), (\Ba',\Bz'))=\varinjlim
  \mathsf{H}/(\mathfrak{m}_{(\Ba',\Bz')}^N  \mathsf{H}+
  \mathsf{H}\mathfrak{m}_{(\Ba,\Bz)}^N)\]
with the obvious composition by multiplication.  As an inverse limit,
this Hom-space has a natural induced topology.  

\begin{theorem}[\mbox{\cite[Th. 17]{FOD}}]\label{th:FOD}
The category of Dunkl-Opdam modules is equivalent to the
representations of the category $\mathcal{H}$ which are continuous in
the discrete topology, via a functor sending the module $M$ to the
representation $(\Ba,\Bz)\mapsto W_{\Ba,\Bz}(M)$.   
\end{theorem}
%Consider the quotient $P_{\Ba,\Bz}^{(N)}=\mathsf{H}/\sum (\mathsf{H} (u_i-
 % a_i)^N+\mathsf{H} (t_i-z_i))$.  Note that this is itself a weight module.
 % These modules form a projective system, whose projective limit
 % $P_{\Ba,\Bz}=\varprojlim P_{\Ba,\Bz}^{(N)}$
 % represents the functor $W_{\Ba,\Bz}$.

Note that this category has a ``polynomial representation'' induced by
the representation of $\mathsf{H}$ on $\mathcal{U}$; this sends
\[(\Ba,\Bz)\mapsto \varinjlim \mathcal{U}/\mathfrak{m}_{(\Ba,\Bz)}^N\cong
\bbK[[U_1-a_1,\dots, U_n-a_n]].\]
This module does not have the discrete topology, and thus does not have a corresponding
Dunkl-Opdam module.  Since the action of $\mathsf{H}$ on $\mathcal{U}$
is faithful, the same is true of the action of $\mathcal{H}$ on the
completions.

Note that the extended affine Weyl group $S_n\ltimes \Z^n$ acts on $\tilde{D}^n$ by
permutations and translations sending \[(\Ba,\Bz)\mapsto
((a_1+m_1,\dots, a_n+m_n),(\zeta^{m_1}z_1,\dots, \zeta^{m_n}z_n)).\]
Two pairs lie in the same orbit if and only if their images in $D$
agree up to permutation of the entries.     For purposes of
understanding this action, it's useful to extend $\Ba$ and $\Bz$ to
arbitrary integers via $a_i=a_{i-n}-1$, and $z_i=z_{i-n}\zeta^{-1}$.

 For two pairs
  $(\Ba,\Bz)$ and $(\Ba',\Bz')=w\cdot (\Ba,\Bz)$ with $w$ in the
  extended affine Weyl group, we let $\xi(\Ba,\Bz,w)$ be
  the straight-line diagram connecting these loadings.     
  
  It's worth noting how these diagrams look for various values of $\Ba,\Bz$ and $w$.  If $w=r_m$, then this straight line diagram $\xi(\Ba,\Bz,r_m)$ moves the strand corresponding to $(a_m,z_m)$ to the right by $\epsilon$ and that for $(a_{m+1},z_{m+1})$ to the left.  This will result in a diagram which is the same up to isotopy, unless:
  \begin{enumerate}
  	\item If $\Upsilon(a_m)= \Upsilon(a_{m+1})$, then the resulting strands will cross.
  	\item If $\Upsilon(a_m-k\ell)= \Upsilon(a_{m+1}) $ then the
        the $m$th strand crosses the ghost of the $m+1$st strand
        moving rightward.
 	\item If $\Upsilon(a_m+k\ell)= \Upsilon(a_{m+1}) $  then   the $m+1$st strand crosses the ghost of the $m$th strand
        moving leftward.
  \end{enumerate}

  The diagram $\xi(\Ba,\Bz,\nu)$ moves each strand $\epsilon$ steps to
  the left, except that corresponding to $a_n$, which moves
  $1-(n-1)\epsilon$ steps to the right; not that this ensures that
  this strand does not cross any strands with the same label, nor the
  ghost of any with adjacent labels.  
Similarly,
  $\xi(\Ba,\Bz,\nu^{-1})$ pushes all strands $\epsilon$ units to the
  right, except that for $a_1$, which moves $1-(n-1)\epsilon$ units to
  the left. Unlike diagrams coming from elements of $S_n$, these can
  create red and black crossings.  

  Let
  \begin{equation}\label{eq:theta}
 \theta_m=(u_m-u_{m+1})r_m-k\ell \pi_{m,m+1}.
  \end{equation}

  \begin{lemma}\label{lem:W-iso}
 There is a fully faithful functor 
  \begin{equation}
\Xi\colon   \widehat{R}_D\to \mathcal{H}.\label{eq:limit-iso}
\end{equation}
such that  $\Xi$ sends the loading $e(\Ba,\Bz)$ to the object
$({\Ba,\Bz})$.  On morphisms, the dot $y_m e(\Ba,\Bz) $ on the strand corresponding to
$(a_m,z_m)$ is sent to
\begin{equation}
\Xi (y_m e(\Ba,\Bz)  )=(u_m-a_m)
e(\Ba,\Bz),\label{eq:dot-image}
\end{equation}
and we have that  
$\Xi(\xi(\Ba,\Bz,r_m))=
e(\Ba,\Bz)  r_m$ if $ z_m\neq z_{m+1}$ and if $z_m= z_{m+1}$, then 
\begin{equation}
\Xi(\xi(\Ba,\Bz,r_m))=
\begin{cases}
{e} (\Ba,\Bz) \frac{1}{u_m-u_{m+1}-k\ell}\theta_m &
a_{m}- k\ell\neq
 a_{m+1}\neq a_m\\
 e(\Ba,\Bz) \theta_m&
a_{m}-k\ell=
 a_{m+1}\neq a_m\\
e(\Ba,\Bz) \frac{1}{u_{m+1}-u_m+k\ell}(r_m-1) &
a_{m}- k\ell\neq
 a_{m+1}=a_m\\
 e(\Ba,\Bz) (1-r_m)&
a_{m}-k\ell=
 a_{m+1}=a_m.
\end{cases}\label{eq:xi-image}
\end{equation}
 Furthermore,
 \begin{equation}
\Xi(\xi(\Ba,\Bz,\nu))=
 \begin{cases}
   \sigma & a_n=p(z_n)\\
   \sigma\frac{1}{u_n-p(z_n)} & a_n\neq p(z_n)
 \end{cases}\qquad \qquad 
\Xi(\xi(\Ba,\Bz,\nu^{-1}))=\tau.\label{eq:sigma-tau-xi}
\end{equation}

  \end{lemma}
In the formulas above, we have used that if $f$ is a $n$-variable polynomial such that
$f( a_1,\dots,  a_n)\neq 0$, then $f(u_1,\dots, u_n)
e(\Ba,\Bz)$ can be inverted, using the geometric series. \begin{proof}
  First, note that the space
  $e(\Ba',\Bz') \cdot {R}_D\cdot e(\Ba,\Bz)$ has a basis over 
  polynomials in the dots which is in bijection with the elements of
  the extended affine Weyl group sending $(\Ba,\Bz)$ to
  $(\Ba',\Bz') $.  Writing a reduced expression of this element, times
  a power of the length 0 rotation shows how to write this basis
  vector (modulo those corresponding to shorter elements of the Weyl
  group) as a product of straight-line diagrams. More precisely,
  we see that $\widehat{R}_D$ is generated over the dots by the diagrams
  $\xi(\Ba,\Bz,r_m)$ and $ \xi(\Ba,\Bz,\nu^{\pm})$.

  We will thus define $\Xi$ by describing the images of these
  elements.  The algebra $R_D$ has a polynomial representation of the weighted KLR algebra
  introduced in \cite[Prop. 2.7]{WebwKLR}; in the categorical
  framework, we can think of this as a functor that sends each loading
  to the polynomial ring $\bbK[Y_1,\dots, Y_n]$.  After completion, we
  obtain an action of $\widehat{R}_D$ that sends each loading to
  $\bbK[[Y_1,\dots, Y_n]]$.  We'll compare polynomial representations
  by using the isomorphism of this ring to $\bbK[[U_1-a_1,\dots,
  U_n-a_n]]$ which sends $Y_i\mapsto U_i-a_i$.  

  Note first that this is compatible with \eqref{eq:dot-image}.
  
In order to calculate the images of $\xi(\Ba,\Bz,r_m)$ and $ \xi(\Ba,\Bz,\nu^{\pm})$, note that the action of $\theta_m$ and $r_m-1$ in the polynomial representation can be
described as:
\begin{align*}
  \theta_m\cdot f&=(u_m-u_{m+1}-k\ell \pi_{m,m+1})  f^{r_m}\\
(r_m-1) \cdot f &= \frac{u_{m}-u_{m+1}-k\ell\pi_{m,m+1}}{u_{m}-u_{m+1}}(f^{r_m}-f)
\end{align*}
The formulas
of \eqref{eq:xi-image} show that:
\[\Xi(\xi(\Ba,\Bz,r_m))\cdot f{e}(\Ba',\Bz') =
\begin{cases}
f^{r_m}{e}(\Ba,\Bz) &
a_{m}-k\ell\neq
 a_{m+1}\neq a_m\\
 (u_m-u_{m+1}-k\ell)  f^{r_m}{e}(\Ba,\Bz)&
a_{m}-k\ell=
 a_{m+1}\neq a_m\\
\frac{f^{r_m}-f} {u_{m+1}-u_m}{e}(\Ba,\Bz)&
a_{m}- k\ell\neq
 a_{m+1}=a_m\\
( {f-f^{r_m}}) {e}(\Ba,\Bz)&
a_{m}-k\ell=
 a_{m+1}=a_m.
\end{cases}
\] 
The four cases in \eqref{eq:xi-image} correspond to:
\begin{enumerate}
\item There are only crossings in the diagram $\xi$ that act
  trivially on the polynomial representation.
\item There is a ghost crossing in $\xi$ corresponding to an arrow
  $\Sigma(a_{m+1},z_{m+1})\to \Sigma(a_{m},z_{m})$ in $D$ where the strand moves left to right.
\item There is a crossing of strands with the same label
  $\Sigma(a_{m},z_{m})=\Sigma(a_{m+1},z_{m+1})$, but no ghost crossing.
\item There is both a strand and a ghost crossing, corresponding to a
  loop at $\Sigma(a_{m},z_{m})=\Sigma(a_{m+1},z_{m+1})$.  
\end{enumerate}
Thus, these match the formulae of \cite[Prop. 2.7]{WebwKLR}. 

In the case of $ \xi(\Ba,\Bz,\nu^{\pm})$, this same correspondence is
easily confirmed. The straight line diagram $ \xi(\Ba,\Bz,\nu)$:
\begin{itemize}
\item only has a ghost
  crossing with an adjacent label if
  $\Upsilon(a_n)> \Upsilon(a_m- k\ell) > \Upsilon(a_n)+1$ for some
  $m$, which is impossible if $a_n$ and $a_m$ lie in the same
  component of $D$ (since then they would differ by a multiple of
  $k\ell$), and
\item only has a red/black crossing if
  $\Upsilon(a_n)\leq \Upsilon(p(z_n)) <\Upsilon(a_n)+1$, but this
  red/black crossing only has an interesting action if $a_n=p(z_n)$.
  Note that in this case, if $z_n=\zeta^m$, we have that the label on
  the corresponding strand is $\Sigma(a_n,z_n)=\cs_m$,
  so this gives the node labeling the
  corresponding red line.
\end{itemize}
Thus, by the formulae of \cite[Prop. 2.7]{WebwKLR}, we have that $ \xi(\Ba,\Bz,\nu^{\pm})$ acts by the identity
unless $p(z_n)=a_n$, in which case it acts by the identity times a dot
on the strand corresponding to $(a_n,z_n)$.    This matches the
action of the elements on the RHS of \eqref{eq:sigma-tau-xi} under the
action \eqref{eq:sigma-action}.

A similar analysis shows
that under the representation of  \cite[Prop. 2.7]{WebwKLR}, the
diagram  $\xi(\Ba,\Bz,\nu^{-1})$ always acts by the identity.  This
matches with \eqref{eq:tau-action}, completing the proof.

This shows that we have a functor 
$R_D\to \mathcal{H}$, which we wish to show
is fully faithful after completion.   First note that  $\Xi$ intertwines the
grading topology with that on $\mathcal{H}$ on the subalgebras
$\bbK[y_1,\dots, y_n]$.  Since $e(\Ba',\Bz') \cdot
{R}_D\cdot e(\Ba,\Bz)$ is finitely generated as a right module over
$\bbK[y_1,\dots, y_n]$, the grading topology on this space is the same
as that induced by any finite set of generators over $\bbK[y_1,\dots,
y_n]$; similarly, each Hom space in $\mathcal{H}$ is finitely
generated as a right module over the suitable completion of
$\mathsf{U}$, and thus has a similar description of its topology.  This
shows that $\Xi$ induces a continuous functor ${R}_D\to \mathcal{H}$,
which thus extends to the completion.

This functor must be injective on $R_D$, since the polynomial representation
remains faithful after completion by \cite[Lem. 2.5]{WebBKnote}.  On
the other hand, we can easily show that generating morphisms of the
category 
$\mathcal{H}$ lie in the image by inverting the formulas (\ref{eq:dot-image}--\ref{eq:sigma-tau-xi}).
\end{proof} 
Let $\widehat{R}_D\mmod_{\operatorname{fd}}$ be the category of
modules over the algebra $\widehat{R}_D$ such that $e(\Ba,\Bz)M$ is
finite dimensional for all $(\Ba,\Bz)$; if, as in Remark \ref{rem:equiv}, we replace $\widehat{R}_D$
by the Morita equivalent algebra where we take one loading from each
equivalence class, these are genuinely finite-dimensional modules.  Note that these
are precisely the finite-dimensional representations of $R_D$ on which the dots act
nilpotently.  Combining Theorem  \ref{th:FOD} and Lemma
\ref{lem:W-iso}, we find that:
\begin{theorem}\label{th:W-iso}
The functor $\mathsf{W}\colon \mathsf{H}\mmod_D\to \widehat{R}_D\mmod_{\operatorname{fd}}$ sending $M\mapsto
\oplus_{\Ba,\Bz} W_{\Ba,\Bz}(M)$ is an equivalence.
\end{theorem}

\subsection{Category $\cO$}
\label{sec:category-co}

For a fixed choice of parameters $k,\ps_i\in \K$, we let 
\[D=\{ {\cs_i+mk}\mid m\in [-n,n]\}\subset \K/\Z.\]  This is a union of
finite linear quivers if $k\notin \Q$ or $n$ is small; it is a union
of $e$-cycles if $k=a/e$ in reduced form, and $n>e/2$.    Taking the
limit as $n\to\infty$, we just obtain the set $\{
{\cs_i+mk}\mid m\in \Z\}$, which is a union of infinity linear quivers
($A_\infty$) or of $e$-cycles.

The category $\mathsf{H}\mmod_D$ has a natural subcategory
$\mathcal{O}^+$ consisting of finitely generated modules on which $x_i$ acts nilpotently,
considered by \cite{GGOR}; we can equally well consider $\cO^-$, where
$y_i$ acts nilpotently, which is the Ringel dual of $ \mathcal{O}^+$
by \cite[4.11]{GGOR}.
In \cite[Th. A]{WebRou}, this category is related to a quotient of the
weighted KLR algebra: the steadied quotient.  We'll only be
interested in a special case of this notion (which in general depends
on a choice of stability condition).
\begin{definition}
  We'll say that a loading is {\bf unsteady} (for the positive
  stability condition) if there exists a real number
  $\delta\geq \Upsilon(\frac{p(\zeta^i)}{\ell})$ such that a non-empty
  set of points in the loading have $x$-value $>\delta+|\Upsilon(k)|$,
  and all others have $x$-value $\leq \delta$.  

There is also a
  negative stability condition where all signs above are reversed: we
  have $\delta\leq \Upsilon(\frac{p(\zeta^i)}{\ell})$, a non-empty set
  of points have $x$-value $<\delta-|\Upsilon(k)|$, and all others
  have $x$-value $\geq \delta$.
\end{definition}
The quotient of $R_D$ by the two-sided ideal generated by the
idempotents $e(\Ba,\Bz)$ which correspond to unsteady
loadings (for one stability condition) is called the  {\bf
  steadied
quotient}; we denote these by $R_D(\pm)$ for the positive/negative
stability condition.

Note that these algebras have a number of desirable properties: they
are cellular and highest weight (since new edges connected to the same
vertex in $D$ always have different weightings) by
\cite[Th. B]{WebRou}. 
\begin{theorem}\label{th:O-iso}
  The functor $\mathsf{W}$ induces an equivalence $\mathcal{O}^{\pm}\cong R_D(\pm)\mmod$.
\end{theorem}
\begin{proof}
Since the proof is the same in both cases, we consider the case of
$\mathcal{O}^-$.  The pair $(\Ba,\Bz)$ corresponds to an unsteady loading if and only if
there exists $I\subset [1,n]$ and a real number $\delta\leq \Upsilon(-\ps_i)$ for
all $i$ such that $\Upsilon(a_i)<\delta-|\Upsilon(k)|$ if $i\in I$ and
$\Upsilon(a_i)\geq \delta$ if $i\notin I$.  Note that permuting an
element of $I$ past one in $[1,n]\setminus I$ gives an isomorphism
between the corresponding weight functors, so without loss of
generality, we can assume that $I=[1,q]$.  Similarly, we have an
isomorphism of $W_{\Ba,\Bz}\cong W_{\Ba_g,\Bz}$ where
$\Ba_g=(a_1-g\ell,\dots, a_q-g\ell,a_{q+1},\dots, a_n)$, since the
corresponding loadings are connected by a crossingless diagram.    If
$N\in \mathcal{O}$, then $N$ must be killed by $W_{\Ba_g,\Bz}$ for
$g\gg 0$, since the Euler eigenvalues of $N$ are bounded below.  Thus,
$\mathsf{W}(N)$ is killed by $e(\Ba,\Bz)$ for any unsteady loading,
and thus the action on it factors through the steadied quotient.

On the other hand,  any pair $(\Ba,\Bz)$  with $\sum \Upsilon(a_j)$ sufficiently
negative must be unsteady, since if a strand is more than
$n|\Upsilon(k)|$ left of a red line, it must be destabilizing. Thus, if the action on $M$ factors through the steadied
quotient, then $\mathsf{h}(M)$ has Euler eigenvalues which are bounded
below.
Since $\mathsf{h}(M)$ is finitely generated, and the action of $\mathsf{eu}$
is locally finite, this shows that $\mathsf{h}(M)$ lies in category
$\cO$.  
\end{proof} 

Perhaps a few remarks are called for about the match of this result
with \cite[Thm. 4.7]{WebRou}.  Theorem \ref{th:O-iso} is more general,
since it does not assume that $\K=\C$. To recover
\cite[Thm. 4.7]{WebRou}, we consider the case where
$\Upsilon\colon \C\to \R$ is given by taking real part.

This theorem allows us to recover in an interesting way the
classification of modules in category $\cO$.  The best known version
of this classification is due to Ginzburg, Guay, Opdam and Rouqiuer:
\begin{theorem}[\mbox{\cite[Prop. 2.11]{GGOR}}]
  For every simple module $S$ in category $\cO^-$, the subspace $U$ of
  elements with minimal weight under $\mathsf{eu}$ is an irreducible
  module over $G(\ell, 1, n)$ and this describes a bijection between
  simples in $\cO$ and over $ G(\ell, 1, n)$.
\end{theorem}
Of course, simple modules over $G(\ell, 1, n)$ are indexed by
$\ell$-multipartitions with $n$ total boxes, and the corresponding
module over $G(\ell, 1, n)$ has a basis indexed by standard tableaux
on the corresponding Young diagram.   Since there are several notions
of standard tableau on a multi-partition, let us clarify that we just mean a filling with $[1,n]$ which
increases in rows and columns.

This construction is carried out in the style of Vershik and Okounkov
\cite{VO} in work of Pushkarev \cite{Pushkarev} and Ogievetsky and
Poulain d'Andecy \cite{OPdA}.  These papers show that, 
in particular,the subalgebra generated by $t_i$
(denoted $j_i$ in \cite{OPdA}) and the Jucys-Murphy elements
(denoted $\tilde{j}_i$ in  {\it loc.\ cit.}) has simple spectrum, with
elements in the spectrum in canonical bijection with standard tableaux
as discussed above.

In \cite[Prop. 11]{OPdA}, they define a representation $V_\xi$
of $\K \Gamma$
with a basis $v_{\mathsf{S}}$ for tableaux $\mathsf{S}$ of shape $\xi$, if the
entry $c$ is in the $i$th row and $j$th column of the $m$th component,
then $t_c$ acts by the scalar $\zeta^m$, and $(c,c+1)$ acts by switching
$c$ and $c+1$ if these are in different components, and by the Young
normal form if they are in the same component.  These are a complete
list of the irreps.

The most important tool in this construction is the
algebra they denote $\mathfrak{A}_{\ell,n}$ in \cite[\S 3]{OPdA}.
This is 
simply our algebra $\mathcal{D\!O}_n$ under an isomorphism
\[\tilde{x}_m\mapsto \frac{1}{k\ell}\DO_m\qquad x_m\mapsto t_m\qquad \bar{s}_i\mapsto
  (i,i+1).\]  In {\it loc.\ cit.}, the
algebra $\K \Gamma$ is written as a quotient of  $\mathcal{D\!O}_n$ by
setting $\DO_1=0$, but this is not the correct map to use for the
elements of minimal $\mathsf{eu}$-weight in a module.

Of course, $\mathcal{D\!O}_n$ acts on the subspace $U$, and does so
via a quotient map to $\K \Gamma$, but not this most obvious
one. Since $\tau$ acts by 0 on $U$, the product
$\sigma \tau=u_1-p(\zeta^{-1}t_1)-1$  does as well.
Thus, we have unique surjective homomorphism $\eta\colon \mathcal{D\!O}_n\to \K \Gamma$
splitting the usual inclusion and killing the 2-sided ideal generated
by  
$u_1-p(\zeta^{-1}t_1)-1$.  
In particular, if we
have a weight $(\Ba,\Bz)$ that appears in $V_\xi$ with $z_1=\zeta^m$ for
$m\in [0,\ell-1]$, then
$a_1=p(\zeta^{m-1})+1$.

Making small changes in arguments of \cite[\S 4]{OPdA}, we can see that the weights
of $V_\xi$ correspond to the tableaux $\mathsf{S}$ of shape $\xi$ as follows:
\begin{lemma}
If the entry $c$ is
   in the $i$th row and $j$th column of the $m$th component, then
   $\DO_c$ and $t_c$ act in the vector $v_{\mathsf{S}}$ by the scalars
   $a_c=p(\zeta^{m-1})+1+k\ell (j-i)$ and $z_c=\zeta^m$.    
\end{lemma}

All of
   these weights will give isomorphic idempotents $e(\Ba,\Bz)$ in
   $R_D$, which match the loading $\Bi_\xi$ introduced in \cite[Def.\
   2.11]{WebRou}; of course, we can see directly from the cellular
   structure of \cite[Th.\
   B]{WebRou} that these must be the lowest weights, showing the
   compatibility with the GGOR perspective.
   
Finally, we turn to considering the KZ functor of $\cO^\pm$.  This functor has a categorical interpretation: it is
represented by the sum of all self-dual projectives, with
multiplicities given by the dimensions of simple modules over Hecke
algebras at roots of unity. The functors $W_{\Ba,\Bz}$ are also
represented by projectives and thus it is natural to try to express
the KZ functor in terms of them.  

Choose a fixed lift $\varphi\colon D\to \K$, where
$\Sigma(\varphi(d),1)=d$.  
Choose an integer \[N\gg \max_{\substack{i\in
  [1,\ell]\\ d\in D}}(|\Upsilon(p(\zeta^i)|,|\Upsilon(k)|,|\Upsilon(\varphi(d))|).\]  For each $n$-tuple
$\mathbf{d}=(d_1,\dots, d_n)\in D^n$, let
\[\mathbf{a}_{\mathbf{d}}^\pm=(\varphi(d_1)\mp N,\varphi(d_2)\mp2N,\dots,
\varphi(d_n)\mp nN)\qquad \mathbf{1}=(1,\dots, 1).\]
\begin{theorem}
  The functor $\mathsf{KZ}$ on $\cO^\pm$ is isomorphic to the sum
  $\displaystyle\bigoplus_{\mathbf{d}\in D^n}
  W_{\mathbf{a}_{\mathbf{d}}^\pm,\mathbf{1}}$.  
\end{theorem}
\begin{proof}
As before, the argument is identical for the two different signs, and
so we consider $\cO^-$.
We need only show that there is an isomorphism between the
representing projectives.  For $\mathbf{d}\in D^N$, we can define a
loading which places a dot with label $d_m$ at $x=mN$. Let $e_{s,n}\in
R_D(-)$ be the sum of the idempotents for these loadings.  From the isomorphisms of Theorems
\ref{th:W-iso} and \ref{th:O-iso}, we know that $\oplus_{\mathbf{d}\in D^n}
  W_{\mathbf{a}_{\mathbf{d}}^-,\mathbf{1}}$ corresponds to the
  projective over $R_D(-)$ given by $R_D(-)e_{{s,N}}$.  
The isomorphism \cite[Thm. 4.5]{WebRou} sends this to the idempotent
$e_{D_{s,N}}$ in the notation of \cite[Sec. 2.5]{WebRou}, which
\cite[Thm. 3.9]{WebRou} shows corresponds to the KZ functor.
\end{proof}

The endomorphisms of the functor
$\oplus_{\mathbf{d}\in D^n}
W_{\mathbf{a}_{\mathbf{d}}^\pm,\mathbf{1}}$
are isomorphic to the cyclotomic KLR algebra with $n$ strands
corresponding to the highest weight $\sum_{i=1}^\ell \omega_{\ps_i}$.
Previous work of Brundan and Kleshchev \cite{BKKL} has constructed an
isomorphism of these to the cyclotomic Hecke algebras which naturally
act by monodromy on $\mathsf{KZ}$.

\subsection{The classification of Dunkl-Opdam modules}
\label{sec:classification}

The equivalence of Theorem \ref{th:W-iso} allows us to classify all simple Dunkl-Opdam
modules over $\mathsf{H}$, not just those in category $\cO^\pm$.

 For a general Dunkl-Opdam module, of course, there is no maximal or minimal
weight under $\mathsf{eu}$.  Instead, we must look for some other patterns
within the weights. % For simplicity, we assume that $k\in \Q$.

A {\bf charged segment} is a $g$-tuple (for some $g\leq n$) of elements
$\mathbf{q}=(q_1,\dots, q_g)$ of
$\K/\Z$, which satisfy $q_{i+1}-q_i=k$.  We'll use {\bf lifted
  segment} to mean a similar $g$-tuple $\mathbf{a}$ in $\K$ satisfying $a_{i+1}-a_i=k\ell$
Choose a large negative integer $P\ll 0$,
and let $\Lambda(q_1,\dots, q_g)$ for a charged segment be the unique
lifted segment $(a_1,\dots, a_g)$ of elements of $\K$ such that
$\Sigma(a_i,z_i)=q_i$,  $a_{i+1}-a_i=k\ell$, $z_{i+1}=z_i$, and
$\Upsilon(a_1)$ is minimized subject to  $P\leq \Upsilon(a_i)$; this
means that $P\leq \Upsilon(a_1)< P+1$ if $\Upsilon(k)\geq 0$ and
$P\leq \Upsilon(a_g)< P+1$  if  $\Upsilon(k)\leq 0$.  A
charged multisegment is an $m$-tuple of charged segments.  The size of
a multisegment is the sum of the lengths of the segments.  

As usual, we can associate to any lifted segment $\mathbf{a}$ and
$z\in \mu_\ell(\K)$,  a 1-dimensional representation of
the algebra $\mathcal{D\!O}_g$ by letting
$S_g$ act trivially, the elements $t_i$ act by the scalar $z$  and
$\DO_i$ act by the scalar $a_i$; to a charged segment $\mathbf{q}$, we
associate the 1-dimensional representation for the distinguished lift $\Lambda(q_1,\dots, q_g)$.  

Note that by the usual theory of modules over degenerate affine Hecke
algebras, based on work of Zelevinsky \cite{Zelsegment} and refined
further by Suzuki \cite{SuzukidAHA}, we can
associate a simple $\mathcal{D\!O}_g$  module $L(\mathbf{Q})$  to any
multisegment $\mathbf{Q}$ of size $g$ by inducing up the tensor
product of the 
1-dimensional modules attached to segments ordered, and taking the unique
simple quotient.  Note that we have to be careful about the order of
lifted 
segments; if two lifted segments with the same $z$ of the form
$(a,a+k\ell,\dots)$ and $(a-hk\ell, a-(h-1)k\ell, \dots)$ with $h\in
\Z_{>0}$ appear, they must be in this order in the induction.  

Let \[\mathcal{D\!O}_{g,n-g}=\mathcal{D\!O}_g\otimes
\mathcal{D\!O}_{n-g}\subset \mathcal{D\!O}_n\] be the subalgebra
generated by $t_i,\DO_i$ for all $i\in [1,n]$ and the Young subgroup $S_g\times S_{n-g}$.  Given a multisegment $\mathbf{Q}$ of size $g$ and an
$\ell$-multipartition $\xi$ of size $n-g$, we have a
$\mathcal{D\!O}_{g,n-g}$ module $ L(\mathbf{Q})\otimes V_{\xi}$ by
taking outer tensor of these modules, where $V_{\xi}$ has the $ \mathcal{D\!O}_{n-g}$ module structure via
the homomorphism $\eta$ 
discussed in the previous section.

We can construct a module over $\mathsf{H}_n$
by considering \[\mathcal{M}(\mathbf{Q}, \xi)=\mathsf{H}_n
\otimes_{\mathcal{D\!O}_{g,n-g}}
(L(\mathbf{Q})\otimes V_{\xi}).\]
 Note that that this definition depends on the choice of $P$.  We
 assume from now on that $P<  \Upsilon(p(\zeta^{m}))- 2n|\Upsilon(k\ell)|$.

\begin{lemma}\label{lem:M-quotient}
Every simple Dunkl-Opdam module  $S$ is a quotient of
$\mathcal{M}(\mathbf{Q}, \xi)$ for some $\mathbf{Q}, \xi$.
\end{lemma}
\begin{proof}
For simplicity, we'll assume throughout the proof that $\Upsilon(k)\geq 0$.
  
  By assumption, we have that  $W_{\Ba,\Bz}(S)\neq 0$ for some
  $(\Ba,\Bz)$.  We claim that we can choose  $(\Ba,\Bz)$ so that
  $\Upsilon(a_i)> P$ for all $i$. We'll prove this by induction on the
  sum $\Pi$ of the quantity $P-\Upsilon(a_i)+1$ over the indices $i$
  such that $\Upsilon(a_i)\leq P$.  Obviously, this is 0 if and only
  if $\Upsilon(a_i)> P$ for all $i$.

  Consider the equivalence relation on the indices $[1,n]$ obtained by
  transitive closure of the relation that
  $i\sim j$ if we have that $a_i=a_{j}\pm k\ell$ and $z_i=z_{j}$.
  Note that we have $|\Upsilon(a_i)-\Upsilon(a_j)|\leq
  n|\Upsilon(k\ell)|$ for any $i\sim j$.  We will use several times
  the fact that
  \begin{itemize}
  \item[$(*)$] if two consecutive indices satisfy $i\not\sim i+1$, then
    $\theta_m\colon W_{(\Ba,\Bz)}(S)\to W_{s_i\cdot (\Ba,\Bz)}(S)$ is
    an isomorphism, so we can reorder these without changing whether the weight space is non-zero.
  \end{itemize}
  Let $i$ be the index that minimizes $\Upsilon(a_i)$.  If for any $i$, we have that $\Upsilon(a_i)\leq P$, then we have
  $\Upsilon(a_j)< \Upsilon(p(\zeta^{m}))$ for all $j\sim i$ and all
  $m$.  If we let $j$ be the largest index such that $j\sim i$, then
by $(*)$ we can 
  assume that $j=n$ without loss of generality.  In this case, have
  that $\Upsilon(a_n)< \Upsilon(p(\zeta^{m}))$ for all $m$, so
  $\sigma$ induces an isomorphism $W_{\Ba,\Bz}(S)\cong
  W_{\nu\cdot(\Ba,\Bz)}(S)$.  The weight $\nu\cdot(\Ba,\Bz)$ has
  strictly fewer indices in the equivalence class of $i$, so we can
  reduce to the case where $i=n$.  
  
  In this case, $(\Ba',\Bz')=  \nu\cdot(\Ba,\Bz)$
  has almost all indices the same, but $a_1'=a_n+1$, so either $\Pi$
  has dropped by exactly 1, or we have strictly fewer indices
  $\Upsilon(a_i)\leq P$, in which case $\Pi$ drops by at least 1.

  Thus, after performing this operation finitely many times, we must
  have 
  $\Pi$ drop to $0$.  Thus, we can assume that $\Upsilon(a_i)> P$ for
  all $i$.

  Now assume that $(\Ba,\Bz)$ minimizes $\sum \Upsilon(a_i)$ amongst weights
  satisfying this condition; that is, we minimize the eigenvalue of
  $\mathsf{eu}$ on this weight space.  Consider the intertwiner
  $\tau\colon W_{\Ba,\Bz}(S)\to
  W_{\nu^{-1}\cdot(\Ba,\Bz)}(S)$.  Since the latter weight space has
  lower Euler eigenvalue, either we must have $\Upsilon(a_1)-1\leq P$,
  or this map is 0; the latter can only happen if
  $a_1=p(\zeta^{m-1})+1, z_1=\zeta^m$ for some $m$, since
  $\sigma\tau=u_1-p(\zeta^{-1}t_1)-1$ must act by 0.  That is, we
  must have exactly one of the options:
  \begin{enumerate}
  \item $a_1=p(\zeta^{m-1})+1, z_1=\zeta^m$
  \item $P<\Upsilon(a_1)-1\leq P+1$
  \end{enumerate}
Using $(*)$ again, we see the same is true of any index $i$ such that $i$ is not equivalent to
any lower index.

  Thus, as before, we can decompose the
  indices $[1,n]$ according the equivalence relation $\sim$, and the
  lowest index in every
  equivalence class satisfies exactly one of (1) or (2).  This in turn
  breaks the indices into two classes which we call types (1') and (2'): either they are greater than or
  less than $P+n|\Upsilon(k\ell)|$.    All elements of an equivalence class containing
  an element satisfying (1) will necessarily be of type (1'), and those
  containing an element satisfying (2) will necessarily be of type
  (2').  The fact $(*)$ shows that we can assume that $[1,g]$ consists
  of 
  indices of type (1') and $[g+1,n]$ of type (2').  

  Now, we consider the module over $\mathcal{D\!O}_{n}$ generated by
  $W_{\Ba,\Bz}$, and consider any simple $K$
  $\mathcal{D\!O}_{n}$-submodule of this space; WLOG, we can assume
  this has non-trivial intersection with $W_{\Ba,\Bz}$.  Let $K'$ be
  the subspace in $K$ given by the sum of all weight spaces such that $[1,g]$ consists
  of 
  indices of type (1') and $[g+1,n]$ of type (2'); by assumption, this
  is a non-trivial module over  $\mathcal{D\!O}_{g,n-g}$.  We have an obvious
  map $\mathcal{D\!O}_{n}\otimes_{\mathcal{D\!O}_{g,n-g}}K'\to K$, and  
  applying $(*)$ shows that this is an isomorphism.  In particular
  $K'$ must be a simple  $\mathcal{D\!O}_{g,n-g}$-module, and thus
  $K'\cong L\otimes V$ for $L$ a simple $\mathcal{D\!O}_{g}$-module
  and $V$ a $\mathcal{D\!O}_{n-g}$-module.

  First, we claim that $L=L(\mathbf{Q})$ for some $\mathbf{Q}$.    The
  module $L$ corresponds to some lifted multisegment; let $a$ be first
  entry in one of these lifted segments which maximizes
  $\Upsilon(a)$.  By assumption $\Upsilon(a)>P$.  We can assume that $a+hk\ell$ does not appear as the
  first entry in one of these lifted segments for all $h\in \Z_{>0}$.
  Thus, the subspace $K'$ contains a weight with $a_1=a$; applying
  $\tau$ maps to a weight space with lower Euler eigenvalue, and is an
  isomorphism since the index $a$ is of type (1').  This is only
  possible if $\Upsilon(a)\leq P+1$, so the same is true of the initial
  element of each segment.  This shows that $L$ has the form
  $L(\mathbf{Q})$.

  Now, assume $g<n$.  Using $(*)$ again, we can also write
  $K=\mathcal{D\!O}_{n}\otimes_{\mathcal{D\!O}_{n-g,g}}(V\otimes L)$;
  the fact that $\tau$ acts trivially on any vector in $V\otimes L$ in
  this embedding shows that $V$ is killed by $u_1-p(\zeta^{-1}t_1)-1$,
  and thus must be of the form $V_\xi$ with $\xi$ having $n-g$ boxes.

  Thus, the inclusion of $\mathcal{D\!O}_{g,n-g}$-modules
  $L(\mathbf{Q})\otimes V_\xi\to S$ induces the desired surjection.  
\end{proof}

 Let $c_\xi$ be the eigenvalue of
 $\mathsf{eu}\in \mathcal{D\!O}_{n-g}$ acting on $V_\xi$.
\begin{definition}
  Let $\Delta(\mathbf{Q}, \xi)$ be the quotient of
  $\mathcal{M}(\mathbf{Q}, \xi)$ by the image of any map from
  $\mathcal{M} (\mathbf{Q}', \xi')$ with $\mathbf{Q}'$ of greater size
  than $\mathbf{Q}$ or $c_{\xi'}<c_{\xi}$.
\end{definition}
\begin{remark}
  If $\mathbf{Q}=\emptyset$, then we can easily check that these are
  the Verma modules in category $\cO$.  We should take pains here to
  emphasize that in general, these are {\it not} the standard modules
  of a quasi-hereditary structure on Dunkl-Opdam modules;
  consideration of the special case $n=1$ shows there is no such
  structure.  However, these are the proper standards of a standardly
  stratified structure one can easily derive from the approach of
  \cite[\S 5.4]{Webmerged}.
\end{remark}
In particular, if we just subtract $\ell$ from $P$
and all elements of $\mathbf{Q}$, then the module $\Delta(\mathbf{Q},
\xi)$  will be unchanged.  
\begin{theorem}
 For fixed $P\ll 0$, every simple Dunkl-Opdam module $S$ is the unique simple quotient of
 $\Delta(\mathbf{Q}, \xi)$ for a unique $\mathbf{Q}$ and $\xi$.
\end{theorem}

\begin{proof}
  Consider a simple Dunkl-Opdam module $S$.  By Lemma
  \ref{lem:M-quotient}, we have that $S$ is a quotient of some
  $M(\mathbf{Q},\xi)$, and we can choose $(\mathbf{Q},\xi)$  with $\xi$ having a minimal number of boxes, and
  $c_\xi$ minimal amongst the possible $\xi$ with the minimal number
  of boxes.

  In this case, $S$ is a quotient of $M(\mathbf{Q},\xi)$ but not of
  any of the $M(\mathbf{Q}',\xi')$ whose images we kill to get
  $\Delta(\mathbf{Q}, \xi)$.  Thus, the map to $M(\mathbf{Q},\xi)\to
  S$ must factor through $\Delta=\Delta(\mathbf{Q}, \xi)$.

  Now we must show that $\Delta$ is unique, and has a unique simple quotient.  Let $(\Ba,\Bz)$ be a weight space in $J=L(\mathbf{Q})\otimes V_\xi$.
  Then for any weight $(\Ba',\Bz')$ satisfying $\Upsilon(a_i')>P$, if we let $\{ w_1,\dots, w_k\}$
  be the finite set of elements of $\widehat{W}$ such that
  $w_p\cdot (\Ba,\Bz)=(\Ba',\Bz')$, then $W_{\Ba',\Bz'}(\Delta)$ is spanned
  by $d_kv$ for $d_k$ a sequence of intertwining operators tracing out
  $w_k$ (or equivalently, $\Xi$ applied to the weighted KLR diagram
  $\xi(\Ba,\Bz,w_k)$) and
  $v\in W_{\Ba',\Bz'}(\Delta)\cap J$.  Note that
  all intermediate steps of these interwining operators pass through
  $(\Ba'',\Bz'')$ with
\[P<\min_i (\Upsilon(a_i'), \Upsilon(a_i))\leq   \Upsilon(a_k'')\leq
  \max_i (\Upsilon(a_i'), \Upsilon(a_i)).\]

Now, assume that $\sum a_i=\sum a_i'$, that is, that these have the
same Euler eigenvalue.  If $w_k$ is not in $S_n$, then we can arrange
this sequence of intertwiners so that a $\tau$ appears before a
$\sigma$ using the relations
(\ref{isotope1},\ref{isotope2},\ref{last}).  Thus, this sequence
factors through a weight space with lower Euler eigenvalue that still
satisfies $\Upsilon(a_i'')>P$ for all $i$.  By assumption, this weight
space is zero.

That is, we must have \[W_{\Ba',\Bz'}(\Delta)\subset
\mathcal{D\!O}_{n}\cdot J=J.\]
This shows that $J$ is uniquely characterized
as the sum of the weight spaces in $\Delta$ which minimize
$\mathsf{eu}$ among those with $\Upsilon(a_i)>P$.  Since this space is
a simple $\mathcal{D\!O}_{n}$-module, any submodule $N$ of $\Delta$
with $N\cap J\neq 0$ must have $J\subset N$ and so $N=\Delta$.  That
is, $N$ is proper if and only if $N\cap J=0$; as usual, this implies
that the sum of all proper submodules is proper and $\Delta$ has a
unique simple quotient.  

On the other hand, this also show that $(\mathbf{Q}, \xi)$ can be
reconstructed from this simple quotient by considering the
$\mathcal{D\!O}_{n} $ action on the sum of the weight spaces in
$\Delta$ which minimize $\mathsf{eu}$ among those with $\Upsilon(a_i)>P$. 
\end{proof}

It's worth noting the similarity of this result to that for the trigonometric
Cherednik algebra (also known as degenerate double affine Hecke
algebra) by Suzuki \cite[Cor. 8.3]{SuzukiCherednik}; if we replace the
equations (\ref{sigtau},\ref{tausig}) by $\sigma\tau=\tau\sigma=1$,
then one can check that we get a slight variation on the usual
presentation of the trigonometric Cherednik algebra, and our result
reduces to Suzuki's.

\subsection{Positive characteristic}
\label{sec:posit-char}

Lemma \ref{lem:W-iso} fails as stated if $\K$ is a field of
characteristic $p$; its very statement uses the existence of
$\Q$-linear maps $\K\to \R$.  However, the functor $\mathsf{W}$ and
the general strategy of computing its endomorphisms remain valid.  The
result is quite interesting because of its relationship to the
coherent sheaves on the degree $n$ Hilbert scheme of
$\C^2/(\Z/\ell\Z)$. More precisely, consider the case where
$\K=\mathbb{F}_p$ for $p\nmid \ell$, and $D$ is the (finite) set of
all pairs possible in this field; let
$\mathsf{Coh}_{\operatorname{pun}}(\mathsf{Hilb}^n(\C^2/(\Z/\ell\Z)))$ be the
category of  coherent sheaves on the Hilbert scheme supported on a
formal neighborhood of the punctual Hilbert scheme.  In the case of
$\ell=1$, this is a well-established result of Bezrukavnikov,
Finkelberg and Ginzburg:
\begin{proposition}[\mbox{\cite[Thms. 1.3.2 \& 1.4.1]{BFG}}]
For $p\gg 0$ and $k$ generic, we have that
 \[ D^b(\mathsf{H}\mmod_D)\cong D^b(\mathsf{Coh}_{\operatorname{pun}}(\mathsf{Hilb}^n(\C^2))).\]
\end{proposition}
This result is extended to  $\ell>1$ in \cite{BFwreath}.

We'll discuss the computation of $\End(\mathsf{W})$ in a more general
context in future work \cite{WebcohII}, where we can give more detailed context; the
combinatorial description of this endomorphism algebra is a
cylindrical version of the KLR algebra which has not yet been
introduced in the literature. This modified KLR algebra is actually a
more useful object for algebraic geometers than the Cherednik algebra,
since even in characteristic 0, it appears as the endomorphisms of a
tilting bundle on the Hilbert scheme, and thus can describe all
coherent sheaves, not just those set-theoretically supported on the
punctual Hilbert scheme.

This also fits into a more general
context about Coulomb branches (as discussed in Section
\ref{sec:coulomb-branches}) in characteristic $p$, which we do not
have the space to develop here.

\section{Coulomb branches}
\label{sec:coulomb-branches}

The isomorphism of Theorem \ref{Cherednik-iso} makes it easy to see the relationship between the
cyclotomic Cherednik algebra and quantum Coulomb branches.  Consider
the $GL_n$ representation $V=\mathfrak{gl}_n\oplus (\C^n)^{\oplus
  \ell}$, and consider the BFN space 
\[\Cspace=\big\{ (g(t),v(t))\in GL_n((t))\times_{GL_n[[t]]}V[[t]]\mid
g(t)\cdot v(t)\in V[[t]]\big\}\]
as discussed in  \cite{NaCoulomb,BFN}.  
For an action of $GL_N$ on any space, we will use the term {\bf
  equivariant parameters} to mean the equivariant Chern classes of the
trivial bundle with fiber $\C^N$.  
The BFN space has:
\begin{enumerate}
\item an action of $\C^*$ by loop
  rotation with equivariant parameter $\ell \hbar$;
\item an obvious action of $GL_n[[t]]$; we will identify the Chern
  classes of the tautological bundle for this action with the
  elementary symmetric polynomials $e_i(\mathbf{U})$, and thus the
  Chern roots with $U_i$;
\item an action of $GL_\ell$ on the multiplicity space of
  $\C^n$;  we will identify the   Chern roots of the tautological bundle
with $-s_i+\ell\hbar $;
\item an action of $\C^*$ by scalar multiplication on
  $\mathfrak{gl}_n$ with equivariant parameter $k$.
\end{enumerate}
All of these actions
commute.  We let $G$ be the product of the first two, and $H$ the
product of the last two.    Consider the $G\times H$-equivariant Borel-Moore homology  $\Coulomb=H_*^{G\times H}(\Cspace)$; this algebra is the quantum Coulomb branch of the
gauge theory attached to $V$.

This
algebra acts naturally on the $G\times H$-equivariant
homology of $V[[t]]$, which is the same as that of a point, that is, a
polynomial ring over $\K$ in the equivariant parameters $\hbar,
e_i(\mathbf{U}),e_i(\bps),k$.  

\begin{theorem}
There is an isomorphism of $e\mathsf{H}e$ with the quantum Coulomb
branch $\Coulomb$.  This isomorphism is induced by the  isomorphism 
$\poly^{\Gamma}\otimes \param\cong H^*_{G\times H}(*)$ discussed above.
\end{theorem}
In \cite{dBHOOY2}, the commutative Coulomb branch of the corresponding
gauge theory is described as the cone $\Sym^n(\C^2/(\Z/\ell\Z))$; by
the uniqueness of quantizations shown by Losev \cite{Losq,Losorbit},
we must have that $ \Coulomb$ is isomorphic to $e\mathsf{H}e$, which
is a well-known quantization of this variety.  However, having a
concrete understanding of this isomorphism is of course, more useful,
and more revealing about the structure of both algebras.  Since a
proof of this result was recently given by Kodera-Nakajima
\cite{KoNa}, we will only sketch the isomorphism below.  However, we
believe it is of some independent interest, since this isomorphism is
quite straightforward given the isomorphism of Theorem
\ref{Cherednik-iso}.

Let us prove a slightly stronger (but none the less easier) version of
this theorem.  The BFN space can be replaced by its Iwahori analogue.
Let $I=\{g(t)\in GL_n[[t]] \mid g(0)\in B\}$ be the standard Iwahori
corresponding to the standard Borel $B$ of upper triangular invertible
matrices.  This analogue is defined by:
\begin{align*}
\mathfrak{I}&=\{v(0)\in
\mathfrak{b}\oplus (\C^n)^{\oplus \ell}\mid v(t)\in V[[t]]\}\\
\Cspace'&=\{ (g(t),v(t))\in GL_n((t))\times_{I}\mathfrak{I}\mid
g(t)\cdot v(t)\in \mathfrak{I}\},
\end{align*}
and the quantum Coulomb branch can be replaced by its
Iwahori version $\fCoulomb=H_*^{I\times \C^*}(\Cspace')$; see \cite[\S
4]{BEF} for a more detailed discussion of this variety. Similarly, we
can replace $e\mathsf{H}e$  by $e'\mathsf{H}e'$
where 
\[e'=\frac{1}{\ell^n}\sum_{\mathbf{i}\in (\Z/\ell\Z)^n}t_1^{i_1}\cdots t_n^{i_n}\] is
just the idempotent symmetrizing for the action of $A=(\Z/\ell\Z)^n$.
More generally, for any character $\eta$ of the group $A$, we have an
idempotent  
\begin{equation}
e_\eta=\frac{1}{\ell^n}\sum_{\mathbf{i}\in (\Z/\ell\Z)^n}\eta(t_1^{-i_1}\cdots
  t_n^{-i_n} )t_1^{i_1}\cdots t_n^{i_n}\label{eq:eta-idempotent}
\end{equation}
the idempotent  of the group algebra $\C[A]$ projecting to this
isotypic component.
We let $E_\eta=e_\eta\cdot 1\in \poly$ and $E'=E_1$; this is effectively the same
sum as \eqref{eq:eta-idempotent}, but with the substitution
$t_i\mapsto T_i$.  
Since $\poly\cong \C[U_1,\dots,
U_n]\otimes_{\C}\C[A]$, we have that \[\poly =\bigoplus_{\eta} \C[U_1,\dots,
U_n]E_\eta.\] 

Thus, both algebras $\fCoulomb$ and $e'\mathsf{H}e' $ act naturally on $\poly^A\cong H^*_{I\times
  \C^*}(*)$, identifying the variables $U_i$ with the Euler classes of
the tautological line bundles on the classifying space of $I$.    
\begin{lemma}
There is an isomorphism of $e'\mathsf{H}e'$ with the flag quantum Coulomb
branch $\fCoulomb$.  This isomorphism is induced by the obvious  isomorphism 
$e'\poly=\poly^A\cong H^*_{I\times
  \C^*}(*)$.  
\end{lemma}
This extension is also proven by Braverman-Etingof-Finkelberg \cite[\S
4.2]{BEF} with a similar proof.
\begin{proof}
  In both cases, we have a copy of polynomial multiplication, given by
  the $e'u_i$ in $e'\mathsf{H}e'$ and the Chern classes of
  tautological bundles in $\fCoulomb$.
We also have copies
of $S_n$ which act as in dAHA.  In
  $\fCoulomb$, this is given by the pullback of the action of $S_n$ on
  the Springer sheaf.  
Finally, the shift element  $e'y_n^{\ell-1}\tau e'$ agrees with the
shift 
correspondence \[X_{\tau}=\{(V_\bullet,V'_{\bullet}\mid
V_i=V'_{i+1}\}\] and 
$e' x_1^{\ell-1}\sigma e'$ agrees with the
correspondence \[X_{\sigma}=\{(V_\bullet,V'_{\bullet}\mid V_i=V'_{i-1}\}\]
To see that these act the same way, we need only check their
commutation with $u_i$, as in (\ref{u-sig}--\ref{u-tau}), and that
they act correctly on the unit.  The commutation is clear, since the
shift correspondence simply reindexes the tautological line bundles.  

The element $e'y_n^{\ell-1}\tau e'$
and $[X_{\tau}]$ both send $E'$ to $E'$.  We claim that the element $e'
 x_1^{\ell-1}\sigma e'$ sends $E'$ to
\begin{multline}
 (U_1+\hbar -p(\zeta^{-1}))\cdots (U_1+(\ell-1)\hbar-
  p(\zeta)) (U_1+\ell\hbar -p(1))E'\\=(U_1-s_{\ell-1}+\ell \hbar)
  \cdots (U_1-s_1+\ell\hbar) E'.\label{eq:push1}
\end{multline}
In order to do this computation, we have to leave $\poly^A$, and
consider elements of $\poly$ transforming over another character
$\eta\colon A\to \C^*$.  
Consider the character $\eta_i(t_j)=\zeta^{\delta_{ij}}$; note that $e_{\eta\eta_i}x_i=x_ie_{\eta}$.

Recall that $x_1=\sigma\upsilon_1^{-1}$.  Note that
\begin{equation}\label{eq:upsilon}\upsilon_1^{-1}t_1=t_n\upsilon_1^{-1}\qquad
  \upsilon_1^{-1}u_1=u_n\upsilon_1^{-1}+a
\end{equation}
where $a$ is a diagram given by permutations of length $<n-1$.
Thus combining \eqref{eq:upsilon} with (\ref{u-sig}) and
(\ref{t-sig}), if we have a polynomial $f(u_1,t_1)$, then \[x_1\cdot
  f(U_1,T_1)e_{\eta} =f(U_1+\hbar,\zeta^{-1}T_1) (U_1+\hbar
  -p(\zeta^{-1})) E_{\eta_i\eta} +(1-e_{\eta_i\eta}) \cdot a'(f)\] for a correction term $a'(f).$

Now, let us apply this to the proof of \eqref{eq:push1}.  
First, note that $\sigma\cdot E' = (U_1+\hbar -p(\zeta^{-1})) E_{\eta_1}$.  Thus, we have that:
\begin{align*}
 e'x_1^{\ell-1}\sigma\cdot  E'&=e'x_1^{\ell-1} \cdot (U_1+\hbar -p(\zeta^{-1}))E_{\eta_i}\\
  &= e' x_1^{\ell-2}\cdot  (U_1+2\hbar -p(\zeta^{-2})) (U_1+\hbar
    -p(\zeta^{-1})) E_{\eta_i^2} + \\
  &\qquad \qquad e'x_1^{\ell-2}(1-e_{\eta_i^2}) \cdot  a'(u_1+\hbar-p(\zeta^{-1}))
\end{align*}
Since $e'x_1^{\ell-2}(1-e_{\eta_i^2})=0$, this correction term vanishes.  Applying this inductively, we find that 
\begin{align*}
 e'x_1^{\ell-1}\sigma E'
  &= e' x_1^{\ell-2}\cdot  (U_1+2\hbar -p(\zeta^{-2})) (U_1+\hbar -p(\zeta^{-1})) E_{\eta_i^2} \\
  &= e' x_1^{\ell-3} \cdot (U_1+3\hbar -p(\zeta^{-3})(U_1+2\hbar -p(\zeta^{-2})) (U_1+\hbar -p(\zeta^{-1})) E_{\eta_i^3} \\
  &\hspace{2mm}\vdots\\
   &=(U_1+\hbar -p(\zeta^{-1}))\cdots (U_1+(\ell-1)\hbar-
  p(\zeta)) (U_1+\ell\hbar -p(1))E'
\end{align*}
This shows equation \eqref{eq:push1}.  

On the other
hand, $[X_{\sigma}]\cdot 1$ is the class of the subspace of flags such
that $\rho\cdot v(t)\in V[[t]]$ where 
\[\rho=
\begin{bmatrix}
  0 & 1 &0&\cdots& 0 \\
  0& 0& 1&\cdots& 0\\
  0& 0& 0&\cdots& 0\\
\vdots & \vdots & \vdots & \ddots &\vdots \\
  t^{-1}& 0& 0&\cdots& 0\\
\end{bmatrix}.
\]
The obstruction to this is the constant term of the first component of
$v(t)$.  This is a section of $\ell$ copies of the tautological bundle
on the affine Grassmannian, which transform according to the standard
representation of $GL_\ell$, and trivially with respect to the loop $\C^*$.  Thus, $[X_{\sigma}]\cdot 1$ is just the
Euler class of this bundle, which agrees with \eqref{eq:push1} by the
convention we have chosen for Chern roots.
This completes the proof that we have a map $e'\mathsf{H}e'\to
\fCoulomb$.   

We note that $\Cspace'$ has a cell
decomposition pulling back the Schubert decomposition, and this map
hits the fundamental class of each cell.   Using the shift elements
constructed above, we see that the map from $e'\mathsf{H}e'$ hits the
classes of 
Schubert cells for all simple reflections.  Multiplying the classes of
the simple reflections in the reduced decomposition of an element of
the Weyl group hits the class of the corresponding Schubert cell, plus
those of shorter length, by a standard argument (see, for example,
\cite[Lemma 3.13]{SWschur}).  Thus, the map is surjective, and the
proof is completed.
\end{proof}

\bibliography{./gen}
\bibliographystyle{amsalpha}
\end{document}